\documentclass[10pt]{amsart}

\usepackage{hyperref}
\usepackage{amssymb,amsmath,amsthm,amsfonts,amsaddr}
\usepackage{enumerate,comment}

\usepackage{tikz}
\usetikzlibrary{positioning}
\tikzset{world/.style=circle,inner sep=2pt,outer sep=3pt}

\newcommand{\alg}{\mathbf}
\newcommand{\class}{\mathsf}
\newcommand{\set}[2]{\{ #1 \mid #2 \}}
\newcommand{\tuple}{\overline}
\newcommand{\equals}{\approx}
\newcommand{\gequals}{\geq}
\newcommand{\assign}{\mathrel{:=}}
\newcommand{\pair}[2]{\langle #1, #2 \rangle}
\newcommand{\bs}{\backslash}
\newcommand{\into}{\hookrightarrow}

\newcommand{\down}{{\downarrow}}
\newcommand{\up}{{\uparrow}}
\newcommand{\idcon}{\Delta}
\newcommand{\iso}{\cong}

\newcommand{\Z}{\mathbb{Z}}
\newcommand{\Zm}{\Z_{-}}
\newcommand{\Zp}{\Z_{+}}

\newcommand{\Nfive}{\alg{N}_{\alg{5}}}
\newcommand{\Mthree}{\alg{M}_{\alg{3}}}
\newcommand{\Two}{\alg{2}}

\newcommand{\minustwo}{\mathsf{-2}}
\newcommand{\minusone}{\mathsf{-1}}
\newcommand{\zero}{\mathsf{0}}
\newcommand{\one}{\mathsf{1}}
\newcommand{\coone}{\overline{\mathsf{1}}}

\DeclareMathOperator{\D}{D}
\DeclareMathOperator{\Idl}{Idl}
\DeclareMathOperator{\Cg}{Cg}
\DeclareMathOperator{\LCg}{LCg}

\DeclareMathOperator{\Con}{Con}
\DeclareMathOperator{\LCon}{LCon}
\DeclareMathOperator{\RCon}{RCon}
\newcommand{\ConLat}{\Con_{\class{Lat}}}
\DeclareMathOperator{\FiOp}{Fi}
\DeclareMathOperator{\NFiOp}{NFi}
\DeclareMathOperator{\FgOp}{Fg}
\newcommand{\Fi}{\FiOp_{\ast}}
\newcommand{\NFi}{\NFiOp_{\ast}}
\newcommand{\Fg}{\FgOp_{\ast}}

\newcommand{\LTheta}{\mathrm{L}\Theta}
\newcommand{\RTheta}{\mathrm{R}\Theta}
\newcommand{\CTheta}{\Theta}
\newcommand{\ThetaPlus}{\Theta_{+}}
\newcommand{\ThetaC}{\Theta_{\mathrm{C}}}

\newcommand{\RL}{\class{RL}}
\newcommand{\CRL}{\class{CRL}}
\newcommand{\pp}{\mathrm{pp}}
\DeclareMathOperator{\Pre}{LPre}

\newtheorem{theorem}{Theorem}[section]
\newtheorem{lemma}[theorem]{Lemma}
\newtheorem{fact}[theorem]{Fact}

\newtheorem{definition}[theorem]{Definition}
\newtheorem{corollary}[theorem]{Corollary}

\title[Pointed lattice subreducts of varieties~of~residuated~lattices]{Pointed lattice subreducts of varieties~of~residuated~lattices}
\author{Adam P\v{r}enosil}
\address{Universitat de Barcelona, Departament de Filosofia}
\email{adam.prenosil@ub.edu}
\thanks{We are grateful to Nick Galatos and Peter Jipsen for their comments on an earlier draft of Section~\ref{section: cancellative}. This work was funded by the grant 2021 BP 00212 of the grant agency AGAUR of the Generalitat de Catalunya.}

\begin{document}

\begin{abstract}
  We study the pointed lattice subreducts of varieties of residuated lattices (RLs) and commutative residuated lattices (CRLs), i.e.\ lattice subreducts expanded by the constant $\one$ denoting the multiplicative unit. Given any positive universal class of pointed lattices $\class{K}$ satisfying a certain equation, we describe the pointed lattice subreducts of semi-$\class{K}$ and of pre-$\class{K}$ RLs and CRLs. The quasivariety of semi-prime-pointed lattices generated by pointed lattices with a join prime constant $\one$ plays an important role here. In particular, the pointed lattices reducts of integral (semiconic) RLs and CRLs are precisely the integral (semiconic) semi-prime-pointed lattices. We also describe the pointed lattice subreducts of integral cancellative CRLs, proving in \mbox{particular} that every lattice is a subreduct of some integral cancellative CRL. This resolves an open problem about cancellative CRLs.
\end{abstract}

\maketitle

  In this paper we shall investigate the pointed lattice subreducts of varieties of residuated lattices, i.e.\ subreducts in the signature consisting of the binary lattice operations and the constant $\one$ (the multiplicative unit). The added value of considering pointed lattice subreducts rather than merely lattice subreducts is that we get more fine-grained information about \emph{where} exactly a sublattice can occur.

  While it is well known that every lattice is a subreduct of some residuated lattice (RL) and in fact of some commutative residuated lattice (CRL), the presence of the constant $\one$ imposes some non-trivial constraints on pointed lattice subreducts. For example, it is known that in every subdirectly irreducible CRL the element $\one$ is join irreducible. Because, like every variety, the variety of CRLs is generated as a quasivariety by its subdirectly irreducible algebras, each pointed lattice subreduct of a CRL must therefore lie in the quasivariety of \emph{semi-irreducible-pointed lattices} generated by pointed lattices where $\one$ is join irreducible. Less trivially, we show that the pointed lattice subreducts of all RLs lie in this quasivariety.

  We do not know whether this quasivariety is precisely the class of pointed lattice subreducts of RLs or CRLs, i.e.\ whether every pointed lattice where $\one$ is join irreducible embeds into a RL or a CRL. We currently only have at our disposal a construction which embeds a pointed lattice where $\one$ is join \emph{prime} into a CRL. The distinction between the join irreducibility and the join primeness of~$\one$, however, disappears in pointed lattices which satisfy the equation $\one \wedge (x \vee y) \equals (\one \wedge x) \vee (\one \wedge y)$.

  We shall accordingly focus our attention on the variety of \emph{semi-$\class{K}$} RLs, where $\class{K}$ is an arbitrary positive universal class of pointed lattices validating this equation. These are RLs which are up to isomorphism subdirect products of RLs with a pointed lattice reduct in $\class{K}$. For example, if $\class{K}$ is the class of all integral pointed lattices (where $\one$ is the top element), semi-$\class{K}$ RLs are precisely the integral RLs, while if $\class{K}$ is the class of conic pointed lattices (where $\one \leq x$ or $x \leq \one$ for each~$x$), semi-$\class{K}$ RLs are precisely the semiconic RLs. Our first main result (Theorem~\ref{thm: main}) states that the pointed lattice subreducts of semi-$\class{K}$ RLs and of (semisimple) semi-$\class{K}$ CRLs are up to isomorphism the subdirect products of irreducible-pointed lattices in~$\class{K}$. For example (Corollary~\ref{cor: main}), the pointed lattice subreducts of integral RLs and of (semimple) integral CRLs are precisely the integral semi-irreducible-pointed lattices, while the pointed lattice subreducts of semiconic RLs and of (semisimple) semiconic CRLs are precisely the semiconic semi-irreducible-pointed lattices.

  In addion to the variety of semi-$\class{K}$ RLs, one can also consider the larger variety of \emph{left pre-$\class{K}$} RLs. This variety can be introduced either axiomatically (Definition~\ref{def: pre-k}) or through several other equivalent definitions (Theorem~\ref{thm: pre-k}). For example, if $\class{K}$ is the class of all linear pointed lattices, left pre-$\class{K}$ RLs are precisely the left prelinear RLs. Our second main result (Theorem~\ref{thm: pre-k subreducts}) shows that the pointed lattice reducts of left pre-$\class{K}$ RLs coincide with the pointed lattice reducts of semi-$\class{K}$ RLs. As a consequence, every left prelinear RL has a distributive lattice reduct (Corollary~\ref{cor: prelinear}) and every left preconic RL has a semiconic pointed lattice reduct (Corollary~\ref{cor: preconic}). This improves on existing theorems which only prove that left prelinearity implies distributivity under additional assumptions (e.g.~\cite[Corollary~4.2.6]{RLBook}).

  We also show that the pointed lattice subreducts of integral cancellative CRLs are precisely the integral semi-irreducible-pointed lattices (Theorem~\ref{thm: cancellative}). As a consequence, every lattice is a subreduct of a (semisimple) integral cancellative CRL (Corollary~\ref{cor: cancellative}). This resolves an open problem about the lattice reducts of cancellative CRLs~\cite[Problem~8.2]{CanRLs} (also posed in~\cite[Problem 7, p.~232]{RLBook}).

  Finally, it remains to axiomatize the quasivariety of semi-irreducible-pointed lattices, which plays a prominent role in the above theorems. We show (Theorem~\ref{thm: semi-prime-pointed}) that this quasivariety of \emph{semi-prime-pointed lattices} generated by pointed lattices where $\one$ is join prime is axiomatized by the conjunction of up-distributivity at $\one$:
\begin{align*}
  x \vee y \gequals \one ~ \& ~ x \vee z \gequals \one \implies x \vee (y \wedge z) \gequals \one,
\end{align*}
  and decomposability at $\one$: for all $n \geq 2$ and all $x_{1}, \dots, x_{n} \in \alg{A}$
\begin{align*}
  x_{1} \vee \dots \vee x_{n} = \one \implies \Cg^{\alg{A}} \pair{x_{1}}{1} \cap \dots \cap \Cg^{\alg{A}} \pair{x_{n}}{1} = \idcon_{\alg{A}}.
\end{align*}
  This implication is admittedly not a quasi-equation, but it is equivalent to an infinite set of quasi-equations. An axiomatization of the quasivariety of semi-irreducible-pointed lattices is then obtained by substituting each variable $u$ above by~$\one \wedge u$.

  The paper is structured as follows. In Section~\ref{sec: pointed lattices} we introduce the main classes of pointed lattices studied in this paper and prove some basic observations about them. In Section~\ref{sec: spp} we axiomatize the quasivariety of semi-prime-pointed lattices. This section can be skipped without loss of continuity: the axiomatization will not play any role in the rest of the paper. In Section~\ref{section: rls} we prove the main results of the paper, namely the description of the pointed lattice subreducts of semi-$\class{K}$ and of pre-$\class{K}$ RLs and CRLs. Finally, in Section~\ref{section: cancellative} we describe the pointed lattice subreducts of integral commutative cancellative residuated lattices and therefore also the lattice subreducts of commutative cancellative residuated lattices.

  All of our results apply equally well to bounded pointed lattices and bounded pointed residuated lattices with no substantial changes to their proofs, except for results concerning cancellative residuated lattices (since there is no non-trivial bounded cancellative residuated lattice).

\section{Pointed lattices}
\label{sec: pointed lattices}

  We start by introducing some basic classes of pointed lattices.

\begin{definition}
  A \emph{pointed lattice} is a lattice equipped with a constant $\one$.
\end{definition}

\begin{definition}
  An \emph{ideal} of a lattice is a non-empty downset closed under binary joins, a \emph{filter} is a non-empty upset closed under binary meets. A \emph{$\one$-filter} is a filter which contains $\one$. A \emph{$\one$-proper ideal} of a pointed lattice is an ideal $I$ such that $\one \notin I$. A $\one$-proper ideal $I$ is \emph{prime} in case $a \wedge b \in I$ implies that either $a \in I$ or $b \in I$. A $\one$-filter $F$ is \emph{prime} in case $a \vee b \in F$ implies that either $a \in F$ or $b \in F$.
\end{definition}

  We shall be interested in equational classes and in positive universal classes of pointed lattices. Recall that an \emph{equational class}, also called a \emph{variety}, is a class of algebras in a given signature axiomatized by some set of (universally quantified) equations, a \emph{positive universal class} is axiomatized by some set of (universally quantified) finite disjunctions of equations, and a \emph{universal class} is axiomatized by some set of (universally quantified) finite disjunctions of equations or negated equations. We shall generally omit the universal quantification when writing down universal sentences. Equivalently, varieties are classes of algebras in a given signature closed under $\mathbb{HSP}$, positive universal classes are classes of algebras closed under $\mathbb{HSP}_{\mathrm{U}}$, and universal classes are classes of algebras closed under $\mathbb{ISP}_{\mathrm{U}}$. Algebras which belong to a universal class $\class{K}$ will be called \emph{$\class{K}$-algebras}.

  The following will be our running examples of positive universal classes of pointed lattices. The class of (dually) integral pointed lattices is in fact a variety.

\begin{definition}
  A pointed lattice is \emph{(dually) integral} if $\one$ is its top (bottom) element. An integral pointed lattice will also be called a \emph{unital lattice}. A pointed lattice is \emph{linear} if it satisfies the positive universal sentence
\begin{align*}
  x \leq y \mathrm{~or~} y \leq x.
\end{align*}
  It is \emph{conic} if it satisfies the positive universal sentence
\begin{align*}
  x \leq \one \mathrm{~or~} \one \leq x.
\end{align*}
\end{definition}

  Further examples of positive universal classes of pointed lattices include pointed lattices of \emph{width} at most $n$ (i.e.\ with antichains of size at most~$n$) and pointed lattices of \emph{height} at most $n$ (i.e.\ with chains of size at most~$n$).

  The principal upset (downset) generated by an element $x$ of a lattice is denoted by $\up x$ (by $\down x$). A pointed lattice $\alg{A}$ is thus conic if and only if $\alg{A} = \up \one \cup \down \one$.

\begin{definition}
  The \emph{positive cone} of a pointed lattice $\alg{A}$ is the (dually integral) pointed sublattice $\alg{A}_{+}$ whose universe is the upset $\up \one$, the \emph{negative cone} is the (integral) pointed sublattice $\alg{A}_{-}$ whose universe is the downset $\down \one$.
\end{definition}

\begin{definition}
  Let $\class{K}$ be a universal class of pointed lattices. Then \emph{semi-$\class{K}$} pointed lattices are pointed lattices isomorphic to subdirect products of pointed lattices in $\class{K}$. Equivalently, the class of semi-$\class{K}$ pointed lattices is $\mathbb{ISP}(\class{K})$.
\end{definition}

  Semi-$\class{K}$ pointed lattices algebras form a \emph{quasivariety}: a class of algebras in a given signature axiomatized by a set of (universally quantified) quasi-equations, or equivalently a class closed under $\mathbb{ISPP}_{\mathrm{U}}$. J\'{o}nsson's lemma ensures that if $\class{K}$ is a positive universal class of pointed lattices, then semi-$\class{K}$ pointed lattices form a variety (see~\cite[Theorem~4.5]{Jonsson}).

  The positive universal classes $\class{K}$ introduced above and the corresponding classes of semi-$\class{K}$ pointed lattices are depicted in the following Hasse diagrams:
\vskip 10pt
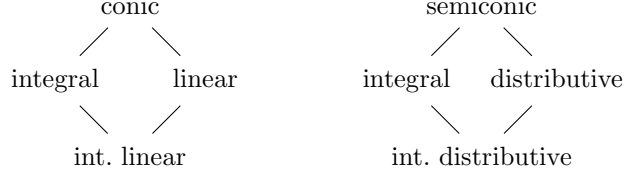
\begin{center}
\begin{tikzpicture}[scale=1, dot/.style={circle,fill,inner sep=2pt,outer sep=2pt}, empty/.style={circle,draw,inner sep=2pt,outer sep=2pt}]
  \node (0) at (0,0) {int.\ linear};
  \node (a) at (-1,1) {integral};
  \node (b) at (1,1) {linear\vphantom{g}};
  \node (c) at (0,2) {conic};
  \draw[-] (0) -- (a) -- (c) -- (b) -- (0);
\end{tikzpicture}
\qquad
\qquad
\begin{tikzpicture}[scale=1, dot/.style={circle,fill,inner sep=2pt,outer sep=2pt}, empty/.style={circle,draw,inner sep=2pt,outer sep=2pt}]
  \node (0) at (0,0) {int.\ distributive};
  \node (a) at (-1,1) {integral};
  \node (b) at (1,1) {distributive\vphantom{g}};
  \node (c) at (0,2) {semiconic};
  \draw[-] (0) -- (a) -- (c) -- (b) -- (0);
\end{tikzpicture}
\end{center}
  Notice that the variety of semilinear pointed distributive lattices is simply the variety of distributive pointed lattices (because each distributive pointed lattice is isomorphic to a subdirect product of one or both of the two pointed expansions of the two-element lattice). The variety of semiconic pointed distributive lattices can be described in a number of equivalent ways.

\begin{fact} \label{fact: semiconic} \label{fact: sp of cones}
  The following are equivalent for each pointed lattice $\alg{A}$:
\begin{enumerate}[(i)]
\item $\alg{A}$ is semiconic.
\item $\alg{A}$ is isomorphic to a subdirect product of $\alg{A}_{+}$ and $\alg{A}_{-}$.
\item $\alg{A}$ satisfies the equations
\begin{align*}
  & \one \wedge (x \vee y) \equals (\one \wedge x) \vee (\one \wedge y), \\
  & \one \vee (x \wedge y) \equals (\one \vee x) \wedge (\one \vee y), \\
  & x \wedge (\one \vee y) \equals (x \wedge \one) \vee (x \wedge y).
\end{align*}
\item $\alg{A}$ satisfies the first two of the above equations and the quasi-equation
\begin{align*}
  \one \wedge x \equals \one \wedge y ~ \& ~ \one \vee x \equals \one \vee y \implies x \equals y.
\end{align*}
\item $\alg{A}$ satisfies the first two of the above equations and does not contain a pointed sublattice isomorphic to the following pointed expansion of the lattice~$\Nfive$:
\vskip 10pt
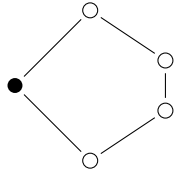
\begin{center}
\begin{tikzpicture}[scale=1, dot/.style={circle,fill,inner sep=2pt,outer sep=2pt}, empty/.style={circle,draw,inner sep=2pt,outer sep=2pt}]
  \node (0) at (0,0) [empty] {};
  \node (a) at (-1,1) [dot] {};
  \node (b) at (1,0.666) [empty] {};
  \node (c) at (1,1.333) [empty] {};
  \node (top) at (0,2) [empty] {};
  \draw[-] (0) -- (a) -- (top) -- (c) -- (b) -- (0);
\end{tikzpicture}
\end{center}
\end{enumerate}
\end{fact}

\begin{proof}
  (i) $\Rightarrow$ (iii): each conic pointed lattice, and therefore each semiconic pointed lattice, satisfies these equations by a simple case analysis (the cases being whether $x$ and $y$ lie in the positive or in the negative cone).

  (iii) $\Rightarrow$ (iv): if $\one \wedge a = \one \wedge b$ and $\one \vee a = \one \vee b$ for some $a, b \in \alg{A}$, then $a = a \wedge (\one \vee a) = a \wedge (\one \vee b) = (a \wedge \one) \vee (a \wedge b) = (b \wedge \one) \vee (b \wedge a) = b \wedge (\one \vee a) = b \wedge (\one \vee b) = b$.

  (iv) $\Rightarrow$ (ii): the following relations $\theta_{\down}$ and $\theta_{\up}$ are congruences of~$\alg{A}$ because of the first two equations:
\begin{align*}
  & \pair{a}{b} \in \theta_{\down} \iff \one \wedge a = \one \wedge b, & & \pair{a}{b} \in \theta_{\up} \iff \one \vee a = \one \vee b.
\end{align*}
  Clearly $\alg{A} / \theta_{\down}$ is an integral pointed lattice and $\alg{A} / \theta_{\up}$ a dually integral one. If $\alg{A}$ moreover satisfies the quasi-equation, then $\theta_{\down} \cap \theta_{\up} = \idcon_{\alg{A}}$, so $\alg{A}$ is isomorphic to a subdirect product of its negative and its positive cone.

  (ii) $\Rightarrow$ (i): trivial.

  (iv) $\Leftrightarrow$ (v): the quasi-equation fails in this expansion of $\Nfive$, so if $\Nfive$ embeds into~$\alg{A}$, then the quasi-equation fails in $\alg{A}$. Conversely, if the quasi-equation fails in~$\alg{A}$, there are $a, b \in \alg{A}$ with $\one \wedge a = \one \wedge b$ and $\one \vee a = \one \vee b$ but $a \neq b$. Take $a' \assign a \vee b$ and $b' \assign a \wedge b$. Then $\one \wedge a' = (\one \wedge a) \vee (\one \wedge b) = \one \wedge a = \one \wedge b'$. Similarly, $\one \vee b' = (\one \vee a) \wedge (\one \vee b) = \one \vee a = \one \vee a'$. Since $b' < a'$, we obtain an isomorphic copy of the forbidden expansion of $\Nfive$ in $\alg{A}$.
\end{proof}

  Within each of the positive universal classes introduced above, we shall be interested in algebras where the constant $\one$ is join prime, or at least join irreducible.

\begin{definition} 
  Let $\alg{A}$ be a lattice. An element $a \hskip-1.1pt \in \hskip-1.1pt \alg{A}$ is \emph{join prime} if for all~${x, y \hskip-1.1pt \in \hskip-1.1pt \alg{A}}$
\begin{align*}
  x \vee y \geq a \implies x \geq a \text{ or } y \geq a.
\end{align*}
  It is \emph{join irreducible} if for all $x, y \in \alg{A}$
\begin{align*}
  x \vee y = a \implies x = a \text{ or } y = a.
\end{align*}
  It is \emph{meet irreducible} if for all $x, y \in \alg{A}$
\begin{align*}
  x \wedge y = a \implies x = a \text{ or } y = a.
\end{align*}
  It is \emph{completely join prime} if for each family $(x_{i})_{i \in I}$ of elements of $\alg{A}$
\begin{align*}
  \bigvee_{i \in I} x_{i} \geq a \implies x_{i} \geq a \text{ for some } i \in I.
\end{align*}
  A \emph{prime-pointed} (\emph{irreducible-pointed}) \emph{lattice} is a pointed lattice such that the element $\one$ is join prime (join irreducible). Given a positive universal class of pointed lattices $\class{K}$, the class of all prime-pointed lattices in $\class{K}$ will be denoted by $\class{K}_{\pp}$.
\end{definition}

  Contrary to common practice, it will be convenient to count the bottom element of a lattice as join irreducible and (completely) join prime, so that the classes of irreducible-pointed and prime-pointed lattices are closed under subalgebras.

  Non-trivial integral prime-pointed lattices are precisely the pointed lattices of the form $\alg{A} \oplus \one$ for some lattice $\alg{A}$: they are obtained from $\alg{A}$ by appending a new top element $\one$. Each lattice is thus a subreduct of an integral prime-pointed lattice.

  Each join prime element is join irreducible, and the converse implication holds in distributive lattices. More importantly for what follows, a conic (in particular, an integral) pointed lattice is prime-pointed if and only if it is irreducible pointed.

\begin{fact} \label{fact: semi-prime-pointed}
  The following are equivalent for each pointed lattice $\alg{A}$:
\begin{enumerate}[(i)]
\item $\alg{A}$ is (semi-)irreducible-pointed.
\item $\alg{A}_{-}$ is (semi-)irreducible-pointed.
\item $\alg{A}_{-}$ is (semi-)prime-pointed.
\end{enumerate}
  In case $\alg{A} \vDash \one \wedge (x \vee y) \equals (\one \wedge x) \vee (\one \wedge y)$, these conditions are equivalent to:
\begin{enumerate}[(i)]
\setcounter{enumi}{3}
\item $\alg{A}$ is (semi-)prime-pointed.
\end{enumerate}
\end{fact}

  As a consequence of the above fact, an axiomatization of the quasivariety of semi-prime-pointed lattices immediately yields an axiomatization of the quasivariety of semi-irreducible-pointed lattices: simply substitute each variable $x$ by $x \wedge \one$.

  In general, semi-$\class{K}_{\pp}$ pointed lattices form a subquasivariety of the quasivariety of semi-$\class{K}$ semi-prime-pointed lattices. For example, this holds in case $\class{K}$ is the positive universal class of pointed lattices of height at most $3$: the unital lattice $\Mthree \oplus \one$ ($\Mthree$ being the five-element diamond) is prime-pointed and up to isomorphism a subdirect product of pointed lattices of height at most $3$, but it is not up to isomorphism a subdirect product of prime-pointed lattices of height at most $3$.

  These two quasivarieties do, however, often coincide. This happens in particular when $\class{K}$ is one of our running examples of integral or conic pointed lattices.

  In the following, given a unital lattice $\alg{A}$, we use $\alg{A} \oplus \one$ to denote by unital expansion of the lattice $\alg{A} \oplus \one$.

\begin{fact} \label{fact: semi-kpp}
  Let $\class{K}$ be a positive universal class of conic pointed lattices such that $\alg{A} \in \class{K}$ implies $\alg{A}_{-} \oplus \one \in \class{K}$. Then the semi-$\class{K}_{\pp}$ pointed lattices coincide with the semi-$\class{K}$ semi-prime-pointed lattices.
\end{fact}

\begin{proof}
  The left-to-right inclusion is trivial. Conversely, let $\alg{A}$ be a semi-$\class{K}$ semi-prime-pointed lattice. Because $\alg{A}$ is semiconic, by Fact~\ref{fact: semiconic} it is up to isomorphism a subdirect product of $\alg{A}_{-}$ and $\alg{A}_{+}$. The positive cone $\alg{A}_{+}$ is also semi-$\class{K}$, so up to isomorphism it is a subdirect product of dually integral and therefore prime-pointed $\class{K}$-algebras. If $\alg{A}$ is dually integral, we are done. Otherwise, the negative cone $\alg{A}_{-}$ is a non-trivial integral semi-$\class{K}$ prime-pointed lattice, so it has the form $\alg{B} \oplus \one$ for some lattice $\alg{B}$. Because $\alg{A}_{-}$ is an integral semi-$\class{K}$ pointed lattice, there is an embedding $\alg{A}_{-} \into \prod_{i \in I} \alg{C}_{i}$ where $(\alg{C}_{i})_{i \in I}$ is a family of unital pointed lattices in~$\class{K}$. Composing this embedding with the inclusion $\alg{B} \into \alg{A}_{-}$ yields an embedding of lattices $\alg{B} \into \prod_{i \in I} \alg{C}_{i}$. By assumption, the unital lattices $\alg{C}_{i} \oplus \one$ also lie in $\class{K}$ and clearly they are prime-pointed. We then obtain an embedding $\alg{B} \oplus \one \into \prod_{i \in I} (\alg{C}_{i} \oplus \one)$ which witnesses that $\alg{A}_{-}$ is up to isomorphism a subdirect product of prime-pointed lattices in~$\class{K}$.
\end{proof}

  The quasivarieties of semi-$\class{K}_{\pp}$ and of semi-$\class{K}$ pointed lattices are frequently generated by their finite algebras. This again happens in particular when $\class{K}$ is one of our running examples of integral and conic pointed lattices.

  The \emph{finite embeddability property (FEP)} for a universal class $\class{K}$ states that each finite partial subalgebra of an algebra in $\class{K}$ embeds into a finite algebra in $\class{K}$. More explicitly, given an algebra $\alg{A} \in \class{K}$ and a finite set $X \subseteq \alg{A}$, there is some finite algebra $\alg{B} \in \class{K}$ and an injective map $\iota\colon X \into B$ such that for each $n$-ary operation $f$ in the signature of $\class{K}$ if $f^{\alg{A}}(a_{1}, \dots, a_{n}) = b$ for $a_{1}, \dots, a_{n}, b \in X$, then $f^{\alg{B}}(\iota(a_{1}), \dots, \iota(a_{n})) = \iota(b)$. A universal class $\class{K}$ has the FEP if and only if it is generated as a universal class by its finite algebras~\cite[Theorem~2.2]{HorcikFEP}.

\begin{fact}
  Each class $\class{K}$ of pointed lattices axiomatized by universal sentences in the signature $\{ \vee, \one \}$ has the FEP. Consequently, the quasivariety of semi-$\class{K}$ pointed lattices is generated by finite semi-$\class{K}$ pointed lattices.
\end{fact}

\begin{proof}
  Consider a pointed lattice $\alg{A}$ and a finite subset $X \subseteq \alg{A}$. Let $B \subseteq \alg{A}$ be the closure of $X \cup \{ \one, \bigwedge X \wedge \one \}$ under joins in $\alg{A}$. As a subposet of $\alg{A}$, $B$ is a finite join semilattice with a least element and thus forms a lattice $\alg{B}$. By construction, the join in $\alg{B}$ of each pair of elements $a, b \in \alg{B}$ agrees with the join in $\alg{A}$, as does their meet in $\alg{B}$ in case $a \wedge b \in X$. Because joins in $\alg{A}$ and $\alg{B}$ agree, each universal sentence in the signature $\{ \vee, \one \}$ which holds in $\alg{A}$ is also satisfied by $\alg{B}$.
\end{proof}

  The variety of semi-$\class{K}$ pointed lattices, for a positive universal class $\class{K}$, is always generated as a variety by $\class{K}$. Frequently, we can improve this to being generated by~$\class{K}_{\pp}$. Again, this happens when $\class{K}$ is one of our running examples of integral and conic pointed lattices, and also when $\class{K}$ is the whole variety of pointed lattices.

  Given a pointed lattice $\alg{A}$, the \emph{doubling} of $\alg{A}$ at $\one$ is the pointed lattice $\D(\alg{A})$ obtained by adding a new element $\one_{-}$ below $\one$ so that $a \leq \one_{-}$ in $\D(\alg{A})$ for $a \in \alg{A}$ if and only if $a < \one$, and $a \geq \one_{-}$ if and only if $a \geq \one$. That is, $\D(\alg{A})$ is obtained from $\alg{A}$ by applying the lattice-theoretic doubling construction to the element $\one$ and taking $\one$ to be the upper cover in the doubling.

\begin{fact}
  Let $\class{K}$ be a positive universal class of conic pointed lattices such that $\alg{A} \in \class{K}$ implies $\alg{A}_{-} \oplus \one \in \class{K}$. Then the variety of semi-$\class{K}$ pointed lattices is generated by prime-pointed lattices in $\class{K}$. If $\class{K}$ is closed under doubling at $\one$, then each semi-$\class{K}$ pointed lattice is a homomorphic image of a prime-pointed lattice in $\class{K}$.
\end{fact}

\begin{proof}
  Each pointed lattice $\alg{A} \in \class{K}$ is conic, therefore so is $\D(\alg{A})$. Up to isomorphism, $\D(\alg{A})$ is thus a subdirect product of $\D(\alg{A})_{-} \iso \alg{A}_{-} \oplus \one \in \class{K}$ and $\D(\alg{A})_{+} \iso \alg{A}_{+} \in \class{K}$. But both $\alg{A}_{-} \oplus \one$ and $\alg{A}_{+}$ are prime-pointed, so $\D(\alg{A})$ is in the variety generated by prime-pointed lattices in $\class{K}$. Finally, the homomorphism $h\colon \D(\alg{A}) \to \alg{A}$ such that $h(\one_{-}) = \one$ and $h(a) = a$ for $a \in \alg{A}$ exhibits $\alg{A}$ as a homomorphic image of~$\alg{B}$.
\end{proof}

\begin{fact}
  The variety of all pointed lattices is generated by prime-pointed lattices. More precisely, each pointed lattice is a homomorphic image of a prime-pointed~one.
\end{fact}

\begin{proof}
  Let $\alg{A}$ be a pointed lattice and $P \assign \up \one$ be its positive cone, let $\Two$ be the integral two-element pointed lattice $\zero < \one$, and let
\begin{align*}
  B \assign \set{\pair{a}{\zero}}{a \in \alg{A}} \cup \set{\pair{p}{\one}}{p \in \up \one}.
\end{align*}
  A simple case analysis shows that $B$ is the universe of an algebra $\alg{B} \leq \alg{A} \times \Two$. Clearly $\alg{A}$ is a homomorphic image of $\alg{B}$ via the projection map $\pi\colon \alg{B} \to \alg{A}$, and $\alg{B}$ is prime-pointed: if $\pair{\one}{\one} \leq \pair{a}{u} \vee \pair{b}{v}$, then either $u = \one$ (in which case $a \in \up \one$) or $v = \one$ (in which case $b \in \up \one$), so either $\pair{\one}{\one} \leq \pair{a}{u}$ or $\pair{\one}{\one} \leq \pair{b}{v}$.
\end{proof}

  Finally, we shall need to introduce ideal completions of pointed lattices. The \emph{ideal completion} of a (pointed) lattice $\alg{A}$ is the lattice $\Idl \alg{A}$ of all ideals of $\alg{A}$ ordered by inclusion (equipped with the constant $\down \one$). The pointed lattice $\alg{A}$ embeds into $\Idl \alg{A}$ via the map $a \mapsto \down a$. This embedding is an isomorphism in case $\alg{A}$ is finite, since each ideal of a finite lattice is principal.

  It is well known that the ideal completion of a lattice $\alg{A}$ satisfies the same equations as $\alg{A}$~\cite[9.1]{CrawleyDilworth}. This fact extends to positive universal sentences.

\begin{lemma}[{\cite[proof of 9.1]{CrawleyDilworth}}] \label{lemma: ideal terms}
  Let $\alg{A}$ be a pointed lattice, $t(x_{1}, \dots, x_{n})$ be a pointed lattice term, and $I_{1}, \dots, I_{n}$ be ideals of $\alg{A}$. Then
\begin{align*}
  t^{\Idl \alg{A}}(I_{1}, \dots, I_{n}) = \set{a \in \alg{A}}{a \leq t^{\alg{A}}(i_{1}, \dots, i_{n}) \text{ for some } i_{1} \in I_{1}, \dots, i_{n} \in I_{n}}.
\end{align*}
\end{lemma}

\begin{fact} \label{fact: idl preserves positive universal sentences}
  The ideal completion $\Idl \alg{A}$ of a pointed lattice $\alg{A}$ satisfies the same positive universal sentences as $\alg{A}$.
\end{fact}

\begin{proof}
  $\alg{A}$ is isomorphic to a subalgebra of $\Idl \alg{A}$, so each positive universal sentence valid in $\Idl \alg{A}$ also holds in $\alg{A}$. Conversely, consider a positive universal sentence
\begin{align*}
  \Phi \assign t_{1}(x_{1}, \dots, x_{n}) \equals u_{1}(x_{1}, \dots, x_{n}) \text{ or } \dots \text{ or } t_{k}(x_{1}, \dots, x_{n}) \equals u_{k}(x_{1}, \dots, x_{n}),
\end{align*}
  such that $\Phi$ holds in $\alg{A}$. Let $I_{1}, \dots I_{n}$ be ideals of $\alg{A}$. We claim that there is some $i \in \{ 1, \dots, k \}$ such that $t_{k}(x_{1}, \dots, x_{n}) \equals u_{k}(x_{1}, \dots, x_{n})$ holds cofinally in $I_{1} \times \dots \times I_{n}$, i.e.\ for all $i_{1} \in I_{1}, \dots, i_{n} \in I_{n}$ there are some $j_{1}, \dots, j_{n} \in \alg{A}$ with $i_{1} \leq j_{1} \in I_{1}, \dots, i_{n} \leq j_{n} \in I_{n}$ such that $t_{k}^{\alg{A}}(j_{1}, \dots, j_{n}) = u_{k}^{\alg{A}}(j_{1}, \dots, j_{n})$: otherwise, for each $i \in \{ 1, \dots, k \}$ there would be some $i^{k}_{1} \in I_{1}, \dots, i^{k}_{n} \in I_{n}$ such that $t_{k}^{\alg{A}}(j_{1}, \dots, j_{n}) = u_{k}^{\alg{A}}(a_{1}, \dots, a_{n})$ never holds for $i^{k}_{1} \leq j_{1} \in I_{1}, \dots, i^{k}_{n} \leq j_{n} \in I_{n}$, in which case interpreting $x_{m}$ for $m \in \{ 1, \dots, k \}$ as $i^{m}_{1} \vee \dots \vee i^{m}_{k}$ would falsify $\Phi$ in~$\alg{A}$. Finally, if $t_{k}(x_{1}, \dots, x_{n}) \equals u_{k}(x_{1}, \dots, x_{n})$ holds cofinally in $I_{1} \times \dots \times I_{n}$, then $t_{k}^{\Idl \alg{A}}(I_{1}, \dots, I_{n}) = u_{k}^{\Idl \alg{A}}(I_{1}, \dots, I_{n})$ by Lemma~\ref{lemma: ideal terms}. Thus $\Phi$ holds in $\Idl \alg{A}$.
\end{proof}

\begin{lemma} \label{lemma: join prime in idl}
  If $\one$ is a join prime element of a lattice $\alg{A}$, then $\down \one$ is a completely join prime element of $\Idl \alg{A}$.
\end{lemma}

\begin{proof}
  If $\down \one \subseteq \bigvee \mathcal{I}$ in $\Idl \alg{A}$ for $\mathcal{I} \subseteq \Idl \alg{A}$, then there are ideals $I_{1}, \dots, I_{n} \in \mathcal{I}$ and elements $i_{1} \in I_{1}, \dots, i_{n} \in I_{n}$ such that $\one \leq i_{1} \vee \dots \vee i_{n}$ in $\alg{A}$. Thus $\one \leq i_{k}$ for some $k \in \{ 1, \dots, n \}$ and $\down \one \subseteq I_{k}$, proving that $\down \one$ is completely join prime.
\end{proof}

\section{Semi-prime-pointed lattices}
\label{sec: spp}

  In this section, which can be skipped without loss of continuity, we axiomatize the quasivariety of semi-prime-pointed lattices by a conjunction of two conditions: up-distributivity at $\one$ and decomposability at $\one$. The former is a quasi-equation, while the latter is a conjunction of infinitely many quasi-equations. We also describe relatively finitely subdirectly irreducible algebras in the quasivariety of semiconic semi-prime-pointed lattices.

\begin{definition} \label{def: sd}
  A pointed lattice is said to be \emph{up-distributive at $\one$} if it satisfies the following quasi-equation:
\begin{align*}
  x \vee y \gequals \one ~ \& ~ x \vee z \gequals \one \implies x \vee (y \wedge z) \gequals \one.
\end{align*}
  It is \emph{join semidistributive at $\one$} if it satisfies the following quasi-equation:
\begin{align*}
  x \vee y \equals \one ~ \& ~ x \vee z \equals \one \implies x \vee (y \wedge z) \equals \one.
\end{align*}
\end{definition}

  The term \emph{join semidistributive at $\one$} is taken from~\cite{JKP}. In the unital case, the term \emph{$\one$-distributive} is also used~\cite[p.~133]{Stern}.

\begin{fact} \label{fact: semidistributive}
  The following are equivalent for each pointed lattice $\alg{A}$:
\begin{enumerate}[(i)]
\item $\alg{A}$ is join semidistributive at $\one$.
\item $\alg{A}_{-}$ is up-distributive at $\one$.
\end{enumerate}
  In case $\alg{A} \vDash \one \wedge (x \vee y) \equals (\one \wedge x) \vee (\one \wedge y)$, these conditions are equivalent to:
\begin{enumerate}[(i)]
\setcounter{enumi}{3}
\item $\alg{A}$ is up-distributive at $\one$.
\end{enumerate}
\end{fact}

  A finite pointed lattice is up-distributive at $\one$ if and only if for each $x$ there is a smallest~$y$ such that $x \vee y \geq \one$ (namely, the meet of all $z$ such that $x \vee z \geq \one$). Similarly, a finite pointed lattice is join semidistributive at $\one$ if and only if for each $x \leq \one$ there is a smallest~$y \leq \one$ such that $x \vee y = \one$.

  If a finite unital lattice $\alg{A}$ is join semidistributive at $\one$, then $\one$ has a canonical join representation in $\alg{A}$~\cite{JonssonKiefer}. That is, $\one = a_{1} \vee \dots \vee a_{m}$ for some $a_{1}, \dots, a_{m} \in \alg{A}$ such that if $\one = b_{1} \vee \dots \vee b_{n}$ for some $b_{1}, \dots, b_{n} \in \alg{A}$, then each $a_{i}$ lies below some~$b_{j}$. Clearly each element of such a canonical join representation must be join prime, so $\one$ is a join of join prime elements. We now show, extending this fact beyond the finite and unital case, that in a pointed lattice $\alg{A}$ the $\one$-filter $\up \one$ is an intersection of prime $\one$-filters if and only if $\alg{A}$ is up-distributive at $\one$. (Recall that a \emph{$\one$-filter} is a lattice filter containing $\one$.)

\begin{definition}
  Let $h\colon \alg{A} \to \alg{B}$ be a homomorphism of pointed lattices. Then the \emph{positive kernel} of $h$ is the $\one$-filter $h^{-1}[\up \one^{\alg{B}}] \subseteq \alg{A}$. By extension, the positive kernel of a congruence $\theta$ of $\alg{A}$ is the positive kernel of the quotient map $\alg{A} \to \alg{A} / \theta$.
\end{definition}

  In the following lemma, we allow for the empty intersection of prime filters, which is understood to be the total filter.

\begin{lemma} \label{lemma: prime}
  The following are equivalent for each pointed lattice $\alg{A}$:
\begin{enumerate}[(i)]
\item $\alg{A}$ is up-distributive at $\one$.
\item The positive cone $\up \one$ of $\alg{A}$ is an intersection of prime $\one$-filters.
\item Each $\one$-proper ideal of $\alg{A}$ extends to a prime $\one$-proper ideal.
\item There is a homomorphism $h\colon \alg{A} \to \alg{B}$ into a distributive pointed (unital) lattice $\alg{B}$ whose positive kernel is $\up \one$.
\end{enumerate}
\end{lemma}


\begin{proof}
  (i) $\Rightarrow$ (iii): let $I$ be a $\one$-proper ideal on $\alg{A}$. By Zorn's lemma, $I$ extends to a maximal $\one$-proper ideal~$J$. To see that $J$ is prime, consider $a, b \in \alg{A}$ such that $a, b \notin J$. By the maximality of $J$, there are $j_{1}, j_{2} \in J$ such that $a \vee j_{1} \geq \one$ and $b \vee j_{2} \geq \one$. Taking $j \assign j_{1} \vee j_{2} \in J$, one obtains that $a \vee j \geq \one$ and $b \vee j \geq \one$, so $(a \wedge b) \vee j \geq \one$ by up-distributivity at~$\one$, and thus indeed $a \wedge b \notin J$.

  (iii) $\Rightarrow$ (ii): for each $a \in \alg{A}$ outside the positive cone the downset $\down a$ is a $\one$-proper ideal, therefore it extends to a prime $\one$-proper ideal $I_{a}$. Its complement $F_{a} \assign L - A$ is a prime $\one$-filter such that $a \notin F_{a}$. Thus $\up \one = \bigcap_{a \in \alg{A}} F_{a}$.

  (ii) $\Rightarrow$ (i): suppose that $\up \one = \bigcap_{i \in I} F_{i}$ for some family of prime filters $F_{i}$. If $x \vee y \geq \one$ and $x \vee z \geq \one$, then $x \vee y \in F_{i}$ and $x \vee z \in F_{i}$ for each $i \in I$, so either $x \in F_{i}$ or both $y \in F_{i}$ and $z \in F_{i}$ for each $i \in I$, and thus $x \vee (y \wedge z) \in F_{i}$ for each $i \in I$. Consequently $x \vee (y \wedge z) \geq \one$.

  (ii) $\Rightarrow$ (iv): each prime $\one$-filter $F$ of $\alg{A}$ determines a homomorphism ${h_{F}\colon \alg{A} \to \Two}$ into the two-element unital lattice $\zero < \one$, namely $h_{F}(a) = \one$ if and only if ${a \in F}$. Combining the homomorphisms $h_{F}$ for all prime $\one$-filters $F$ of $\alg{A}$ yields a homomorphism $h\colon \alg{A} \to \Two^{P}$, where $P$ is the set of all prime $\one$-filters of $\alg{A}$. The positive kernel of $h$ is the intersection of the positive kernels of the homomorphisms $h_{F}$ for $F \in P$, which is $\up \one$ by assumption.

  (iv) $\Rightarrow$ (ii): each $\one$-filter of a distributive pointed lattice $\alg{B}$ (in particular, the principal filter $\up \one$) is an intersection of prime $\one$-filters. The same therefore holds for the positive kernel $h^{-1}[\up \one]$ of the homomorphism $h$.
\end{proof}

  Given an algebra $\alg{A}$ and elements $c, d \in \alg{A}$, the congruence generated by a set of pairs $X \subseteq A^{2}$ will be denoted by $\Cg^{\alg{A}} X$, with $\Cg^{\alg{A}} \pair{c}{d} \assign \Cg^{\alg{A}} \{ \pair{c}{d} \}$. The identity congruence on $\alg{A}$ will be denoted by $\idcon_{\alg{A}}$. Recall that an algebra $\alg{A}$ is \emph{subdirectly irreducible}, or \emph{s.i.}\ for short, if for each family of congruences $(\theta_{i})_{i \in I}$ of $\alg{A}$, if $\bigcap_{i \in I} \theta_{i} = \idcon_{\alg{A}}$, then $\theta_{i} = \idcon_{\alg{A}}$ for some $i \in I$. It is \emph{finitely subdirectly irreducible}, or \emph{f.s.i.}\ for short, if this implication holds for $I$ finite. 

\begin{definition}
  A pointed lattice $\alg{A}$ is said to be \emph{decomposable at $\one$} if for all $n \geq 2$ and all $x_{1}, \dots, x_{n} \in \alg{A}$
\begin{align*} \label{eq: theta_n}
  x_{1} \vee \dots \vee x_{n} = \one \implies \Cg^{\alg{A}} \pair{x_{1}}{1} \cap \dots \cap \Cg^{\alg{A}} \pair{x_{n}}{1} = \idcon_{\alg{A}}. \tag{$\theta_{n}$}
\end{align*}
\end{definition}

  That is, if $x_{1} \vee \dots \vee x_{n} = \one$, then $\alg{A}$ has a subdirect embedding into a product of the algebras $\alg{A} / \theta_{i}$ for $\theta_{i} \assign \Cg^{\alg{A}} \pair{x_{i}}{\one}$ via the natural quotient maps $\alg{A} \to \alg{A} / \theta_{i}$.

  Observe that if $\one$ is join irreducible in $\alg{A}$, then $\alg{A}$ is decomposable at $\one$. The converse holds in case $\alg{A}$ is finitely subdirectly irreducible.

  The following lemma shows that pointed lattices decomposable at $\one$ form a quasivariety. Its proof gives an explicit axiomatization of this quasivariety.

\begin{lemma}
  The condition (\ref{eq: theta_n}) is equivalent to a set of quasi-equations.
\end{lemma}

\begin{proof}
  Congruence generation is described by Maltsev's Lemma~\cite[Theorem~4.17]{Bergman}, which (abstracting away from the details) states that there is a family $\Phi$ of sets of equations where each $E \in \Phi$ is a set of equations in the variables $u, v, x, y, \tuple{z}$ (for some finite tuple of variables $\tuple{z}$, whose length depends on the set $E$) such that
\begin{align*}
  \pair{a}{b} \in \Cg^{\alg{A}} \pair{c}{d} \iff \alg{A} \vDash E(a, b, c, d, \tuple{e}) \text{ for some } E \in \Phi \text{ and some tuple } \tuple{e} \in \alg{A}.
\end{align*}
  The condition (\ref{eq: theta_n}) is therefore equivalent to the conjunction of all quasi-equations of the form
\begin{align*}
  x_{1} \vee \dots \vee x_{n} \equals \one ~ \& ~ E_{1}(u, v, x_{1}, \one, \tuple{z}_{1}) ~ \& ~ \dots ~ \& ~ E_{n}(u, v, x_{n}, \one, \tuple{z}_{n}) \implies u \equals v
\end{align*}
  for $n \geq 2$ and $E_{1}, \dots, E_{n} \in \Phi$.
\end{proof}

  The positive kernel of a congruence $\theta$ of $\alg{A}$ is a prime filter if and only if $\alg{A} / \theta$ is a prime-pointed lattice. We now show that every prime filter $F$ of a pointed lattice arises as the positive kernel of a congruence.

\begin{lemma} \label{lemma: theta}
  Let $\alg{A}$ be a pointed lattice and let $F$ be a prime $\one$-filter of $\alg{A}$. Then there is a smallest congruence $\ThetaPlus(F)$ of $\alg{A}$ such that $F$ is the positive kernel of~$\ThetaPlus(F)$, namely
\begin{align*}
  \ThetaPlus(F) \assign \Cg^{\alg{A}} \set{\pair{f \wedge \one}{\one}}{f \in F}.
\end{align*}
\end{lemma}

\begin{proof}
  To prove that $F$ is the positive kernel of $\theta \assign \ThetaPlus(F)$, let $P$ be the subset of $\alg{A}$ corresponding to the positive cone of $\alg{A} / \theta$. Clearly $F \subseteq P$ by the definition of~$\theta$. Conversely, consider an element $p \in P$. That is, $\pair{p \wedge \one}{\one} \in \theta$. Because $F$ is a prime filter, the equivalence relation $\psi$ with two equivalence classes $F$ and $L - F$ is a congruence. Because $\pair{f \wedge \one}{\one} \in \psi$ for each $f \in F$, we have $\theta \subseteq \psi$, so $\pair{p \wedge \one}{\one} \in \psi$. But $\one \in F$, thus $p \wedge \one$ and consequently also $p$ must lie in $F$, proving that $F = P$. Finally, if $\phi$ is a congruence whose positive kernel is $F$, then $\pair{f \wedge \one}{\one} \in \phi$ for each $f \in F$, so $\ThetaPlus(F) \subseteq \phi$.
\end{proof}

  The following quasi-equations for $n \geq 2$ hold in all semi-prime-pointed lattices:
\begin{align*}
  x_{1} \vee \dots \vee x_{n} \geq \one \implies (x_{1} \wedge z) \vee \dots \vee (x_{n} \wedge z) \geq \one \wedge z. \tag{$\alpha_{n}$}
\end{align*}
  In particular, for $n \assign 2$ we get the quasi-equation
\begin{align*} \label{eq: alpha}
  x \vee y \geq \one \implies (x \wedge z) \vee (y \wedge z) \geq \one \wedge z. \tag{$\alpha_{2}$}
\end{align*}

\begin{lemma} \label{lemma: alpha}
  Each pointed lattice which satisfies (\ref{eq: alpha}) is up-distributive at $\one$.
\end{lemma}

\begin{proof}
  Suppose that $a \vee b \geq \one$ and $a \vee c \geq \one$. Then $(a \wedge \one) \vee (b \wedge \one) = \one$ and $(a \wedge \one) \vee (c \wedge \one) = 1$ by (\ref{eq: alpha}), so $\one \wedge b = (a \wedge \one \wedge b) \vee (c \wedge \one \wedge b)$ by a further application of (\ref{eq: alpha}). Consequently, $\one = (a \wedge \one) \vee (b \wedge \one) = (a \wedge \one) \vee (a \wedge \one \wedge b) \vee (b \wedge c \wedge \one) = (a \wedge \one) \vee (b \wedge c \wedge \one)$, and thus $a \vee (b \wedge c) \geq \one$.
\end{proof}

  The converse implication does not hold: the unital lattice expansion of the pentagon lattice $\Nfive$ is up-distributive at $\one$ but fails to satisfy (\ref{eq: alpha}).

  The following theorem tells us that if $\alg{A}$ fails to be prime-pointed, we can prove this by falsifying either up-distributivity at $\one$ or decomposability at~$\one$, whereas if a subdirect decomposition of $\alg{A}$ into prime-pointed lattices exists, we can always take factors of the form $\alg{A} / \ThetaPlus(F)$ with $F$ ranging over prime $\one$-filters.

\begin{theorem} \label{thm: semi-prime-pointed}
  The following are equivalent for each pointed lattice $\alg{A}$:
\begin{enumerate}[(i)]
\item $\alg{A}$ is semi-prime-pointed.
\item $\alg{A}$ is up-distributive at $\one$ and decomposable at $\one$.
\item $\alg{A}$ satisfies (\ref{eq: alpha}) and is decomposable at $\one$.
\item $\bigcap_{i \in I} \ThetaPlus(F_{i}) = \idcon_{\alg{A}}$ for some family $(F_{i})_{i \in I}$ of prime $\one$-filters of $\alg{A}$.
\end{enumerate}
  In case $\alg{A}$ is finite, the above conditions are equivalent to:
\begin{enumerate}[(i)]
\setcounter{enumi}{4}
\item $\alg{A}$ is up-distributive at $\one$ and for each family $(F_{i})_{i \in I}$ of prime $\one$-filters of $\alg{A}$ if $\bigcap_{i \in I} F_{i} = \up \one$, then $\bigcap_{i \in I} \ThetaPlus(F_{i}) = \idcon_{\alg{A}}$.
\end{enumerate}
\end{theorem}

\begin{proof}
  (i) $\Rightarrow$ (ii): clearly prime-pointed lattices are up-distributive at $\one$. They are also decomposable at $\one$: if $x_{1} \vee \dots \vee x_{n} = \one$, then $x_{i} = \one$ for some $i \in \{ 1, \dots, n \}$, so $\Cg^{\alg{A}} \pair{x_{i}}{\one} = \idcon_{\alg{A}}$. Because up-distributivity at $\one$ and decomposability at $\one$ are quasi-equational conditions, they are preserved under $\mathbb{ISPP}_{\mathrm{U}}$.

  (ii) $\Rightarrow$ (i): let $\alg{A}$ be a pointed lattice up-distributive at $\one$ and decomposable at~$\one$. If $\alg{A}$ is dually integral, then it is prime-pointed. We may therefore assume that $\alg{A}$ is not dually integral. Consider distinct elements $a, b \in \alg{A}$. To prove the inclusion, it suffices to find a congruence $\theta$ on $\alg{A}$ such that $\pair{a}{b} \notin \theta$ and the pointed lattice $\alg{A} / \theta$ is prime-pointed.

  To this end, let
\begin{align*}
  I \assign \set{x \in \alg{A}}{x \leq \one \text{ and } \pair{a}{b} \in \Cg^{\alg{A}} \pair{x}{\one}}.
\end{align*}
  Since $\Cg^{\alg{A}} \pair{y}{\one} \leq \Cg^{\alg{A}} \pair{x}{\one}$ if $x \leq y \leq \one$, the set $I$ is a downset. Let $J$ be the ideal generated by $I$ in case $I$ is non-empty, and let $J \assign \down x$ for some arbitrary $x < \one$ in case $I = \emptyset$ (such an element $x$ exists because $\alg{A}$ is not dually integral). Clearly $J$ is a $\one$-proper ideal in the latter case. Decomposability at $\one$ implies that $J$ is a $\one$-proper ideal also in the former case: otherwise $x_{1} \vee \dots \vee x_{n} = \one$ for some $x_{1}, \dots, x_{n} \in I$, so $\Cg^{\alg{A}} \pair{x_{1}}{\one} \cap \dots \Cg^{\alg{A}} \pair{x_{n}}{\one} = \idcon_{\alg{A}}$, but by the definition of $I$ this would imply that $\pair{a}{b} \in \idcon_{\alg{A}}$, i.e.\ that $a = b$. Up-distributivity at $\one$ implies that $J$ extends to a prime $\one$-proper ideal $K$.

  Let $F \assign L - K$. Then $F$ is a prime $\one$-filter, so $\alg{A} / \ThetaPlus(F)$ is a prime-pointed lattice by Lemma~\ref{lemma: theta}. Moreover, $\pair{a}{b} \notin \ThetaPlus(F)$: if $\pair{a}{b} \in \ThetaPlus(F)$, then there is some $f \in F$ such that $\pair{a}{b} \in \Cg^{\alg{A}} \pair{f \wedge \one}{\one}$, since $\ThetaPlus(F)$ is a directed join of such congruences. But then $f \wedge \one \in I \subseteq K$ and $f \wedge \one \notin F$, contradicting $f \in F$ (or indeed contradicting $I = \emptyset$ if $I$ is empty).

  (i) $\Rightarrow$ (iii): each prime-pointed and therefore each semi-prime-pointed lattice satisfies (\ref{eq: alpha}). We have already proved that each semi-prime-pointed lattice is decomposable at $\one$.

  (iii) $\Rightarrow$ (ii): this is Lemma~\ref{lemma: alpha}.

  (i) $\Rightarrow$ (iv): suppose that $\alg{A} \leq \prod_{i \in I} \alg{B}_{i}$ for some family $(\alg{B}_{i})_{i \in I}$ of prime-pointed lattices. Let $F_{i} \assign \pi_{i}^{-1}[\up \one^{\alg{B}_{i}}]$ for $i \in I$, where $\pi_{i}\colon \alg{A} \to \alg{B}_{i}$ is the projection map. Then each $F_{i}$ is a prime $\one$-filter. Moreover, $\ThetaPlus(F_{i}) \leq \ker \pi_{i}$, so $\bigcap_{i \in I} \ThetaPlus(F_{i}) \subseteq \bigcap_{i \in I} \ker \pi_{i} = \idcon_{\alg{A}}$.

  (iv) $\Rightarrow$ (i): the pointed lattice $\alg{A}$ is isomorphic to a subdirect product of the pointed lattices $\alg{A} / \ThetaPlus(F_{i})$ for $i \in I$, which are all prime-pointed by Lemma~\ref{lemma: theta}.

  (ii) $\Rightarrow$ (v): if $\alg{A}$ is finite, then each $\one$-filter $F$ of $\alg{A}$ has the form $F = \up f$ for some $f \leq \one$. But then $\ThetaPlus(F) = \Cg^{\alg{A}} \pair{f}{\one}$, so the implication in (iv) is a special case of decomposability at $\one$.

  (Alternatively, instead of (ii) $\Rightarrow$ (v) we may prove the implication (i) $\Rightarrow$ (v) using an argument analogous to the proof of the implication (ii) $\Rightarrow$ (vi) in Theorem~\ref{thm: semiconic semi-prime-pointed}.)

  (v) $\Rightarrow$ (i): up-distributivity at $\one$ implies that $\up \one = \bigcap_{i \in I} F_{i}$ for some family $(F_{i})_{i \in I}$ of $\one$-filters of $\alg{A}$. Consequently, $\alg{A}$ is a subdirect product of the algebras $\alg{A} / \ThetaPlus(F_{i})$ for $i \in I$, which are prime-pointed by Lemma~\ref{lemma: theta}.
\end{proof}

  As we have already observed, a pointed lattice is semi-irreducible-pointed if and only if its negative cone is semi-prime-pointed, therefore the above theorem immediately yields a description of semi-irreducible-pointed lattices.

  Subdirect products of conic prime-pointed lattices can be described similarly, replacing $\ThetaPlus(F)$ with a larger congruence $\ThetaC(F)$.

\begin{lemma} \label{lemma: thetac}
  Let $\alg{A}$ be a pointed lattice and let $F$ be a prime $\one$-filter of $\alg{A}$. Then there is a smallest congruence $\ThetaC(F)$ of $\alg{A}$ such that $\alg{A} / \ThetaC(F)$ is conic and $F$ is the positive kernel of $\ThetaC(F)$, namely
\begin{align*}
  \ThetaC(F) & \assign \Cg^{\alg{A}} \set{\pair{f \wedge \one}{\one}}{f \in F} \vee \Cg^{\alg{A}} \set{\pair{i \wedge \one}{i}}{i \notin F}.
\end{align*}
\end{lemma}

\begin{proof}
  Clearly $\alg{A}$ is conic by the definition of $\theta \assign \ThetaC(F)$. The proof of that $\alg{A} / \theta$ is a prime-pointed lattice and $F$ is the positive kernel of $\theta$ is the same as in Lemma~\ref{lemma: theta}. If $\phi$ is a congruence whose positive kernel is $F$ such that $\alg{A} / \phi$ is conic, then $\pair{f \wedge \one}{\one} \in \phi$ for each $f \in F$ and $\pair{i \wedge \one}{i} \in \phi$ for each $i \notin F$, so $\ThetaC(F) \subseteq \phi$.
\end{proof}

\begin{theorem} \label{thm: semiconic semi-prime-pointed}
  The following are equivalent for each pointed lattice $\alg{A}$:
\begin{enumerate}[(i)]
\item $\alg{A}$ is semiconic and semi-prime-pointed.
\item $\alg{A}$ is a subdirect product of conic prime-pointed lattices.
\item $\alg{A}$ is semiconic, up-distributive at $\one$, and decomposable at $\one$.
\item $\alg{A}$ is semiconic, decomposable at $\one$, and satisfies (\ref{eq: alpha}).
\item $\bigcap_{i \in I} \ThetaC(F_{i}) = \idcon_{\alg{A}}$ for some family $(F_{i})_{i \in I}$ of prime $\one$-filters of $\alg{A}$.
\end{enumerate}
  In case $\alg{A}$ is finite, the above conditions are equivalent to:
\begin{enumerate}[(i)]
\setcounter{enumi}{5}
\item $\alg{A}$ is up-distributive at $\one$ and for each family $(F_{i})_{i \in I}$ of prime $\one$-filters of $\alg{A}$ if $\bigcap_{i \in I} F_{i} = \up \one$, then $\bigcap_{i \in I} \ThetaC(F_{i}) = \idcon_{\alg{A}}$.
\end{enumerate}
\end{theorem}

\begin{proof}
  (iii) $\Leftrightarrow$ (iv): this follows from the equivalence (ii) $\Leftrightarrow$ (iii) in Theorem~\ref{thm: semi-prime-pointed}.

  (i) $\Leftrightarrow$ (ii): this is an instance of Fact~\ref{fact: semi-kpp}.

  (i) $\Leftrightarrow$ (iii): this is an immediate consequence on Theorem~\ref{thm: semi-prime-pointed}.

  (ii) $\Rightarrow$ (v): let $(\theta_{i})_{i \in I}$ be a family of congruences of $\alg{A}$ such that each $\alg{A} / \theta_{i}$ is a conic prime-pointed lattice, and let $\alg{A}$ be a subdirect product of the pointed lattices $\alg{A} / \theta_{i}$. If $a \neq b$ for $a, b \in \alg{A}$, then there is some $i \in I$ such that $\pair{a}{b} \notin \theta_{i}$. Let $F_{i}$ be the positive kernel of $\theta_{i}$. Since $\alg{A} / \theta_{i}$ is conic, $\ThetaC(F_{i}) \subseteq \theta_{i}$ for each $i \in I$ by Lemma~\ref{lemma: thetac}. Thus $\bigcap_{i \in I} \ThetaC(F_{i}) \subseteq \bigcap_{i \in I} \theta_{i} = \idcon_{\alg{A}}$.

  (v) $\Rightarrow$ (ii): this holds because for each prime $\one$-filter $F$ of $\alg{A}$ the algebra $\alg{A} / \ThetaC(F)$ is conic and prime-pointed by Lemma~\ref{lemma: thetac}.

  (vi) $\Rightarrow$ (ii): proved as in Theorem~\ref{thm: semi-prime-pointed}.

  (ii) $\Rightarrow$ (vi): let $\alg{A}$ be a subdirect product of a finite family of finite conic prime-pointed lattices $(\alg{B}_{j})_{j \in J}$. That is, there are surjective homomorphisms $h_{j}\colon \alg{A} \to \alg{B}_{j}$ such that $\bigcap_{j \in J} \ker h_{j} = \idcon_{\alg{A}}$. Let $G_{i} \assign h_{j}^{-1}[\up \one^{\alg{B}_{j}}]$. By restricting to minimal $\one$-filters in the family $(F_{i})_{i \in I}$, we may assume without loss of generality that the $\one$-filters $F_{i}$ are pairwise incomparable by inclusion. Because $\bigcap_{j \in J} G_{j} = \up \one = \bigcap_{i \in I} F_{i}$ and the finite families $(F_{i})_{i \in I}$ and $(G_{j})_{j \in J}$ consist of prime $\one$-filters, there is for each $F_{i}$ some $j \in J$ with $G_{j} \subseteq F_{i}$ and some $k \in I$ with $F_{k} \subseteq G_{j}$. But then $F_{k} \subseteq F_{i}$, so $k = i$ and $F_{i} = G_{j}$. Similarly, there is for each $G_{j}$ some $i \in I$ such that $G_{j} = F_{i}$. That is, $\set{F_{i}}{i \in I} = \set{G_{j}}{j \in J}$. Consequently, $\bigcap_{i \in I} \ThetaC(F_{i}) = \bigcap_{j \in J} \ThetaC(G_{j}) \subseteq \bigcap_{j \in J} \ker h_{j} = \idcon_{\alg{A}}$ by Lemma~\ref{lemma: thetac}.
\end{proof}

  The relatively (finitely) subdirectly irreducible algebras in the quasivariety of semiconic semi-prime-pointed lattices, i.e.\ by Fact~\ref{fact: semi-kpp} the conic relatively (finitely) subdirectly irreducibles in the quasivariety of semi-prime-pointed lattices, admit a straightforward description in terms of (finitely) subdirectly irreducible lattices.

  Recall that given an algebra $\alg{A}$ in a quasivariety $\class{Q}$, a congruence $\theta$ of $\alg{A}$ is a \emph{$\class{Q}$-congruence} if $\alg{A} / \theta \in \class{Q}$. The algebra $\alg{A}$ is \emph{relatively simple} (relative to $\class{Q}$) if $\alg{A}$ has exactly two $\class{Q}$-congruences. It is \emph{relatively subdirectly irreducible} (relative to $\class{Q}$), or \emph{r.s.i.}\ for short, if for each family $(\theta_{i})_{i \in I}$ of $\class{Q}$-congruences $\bigcap_{i \in I} \theta_{i} = \idcon_{\alg{A}}$ implies that $\theta_{i} = \idcon_{\alg{A}}$ for some $i \in I$, or equivalently if for each embedding $\iota\colon \alg{A} \into \prod_{i \in I} \alg{B}_{i}$ with $\alg{B}_{i} \in \class{Q}$ for each $i \in I$, the map $\pi_{i} \circ \iota\colon \alg{A} \to \alg{B}_{i}$ is an embedding for some $i \in I$, where $\pi_{i}\colon \prod_{i \in I} \alg{B}_{i} \to \alg{B}_{i}$ is the projection map. \emph{Relatively finitely subdirectly irreducible algebras}, or \emph{r.f.s.i.}\ algebras for short, satisfy these conditions for $I$ finite.

  The following lemma about r.f.s.i.\ algebras is a quasivariety analogue of J\'{o}nsson's Lemma, proved in~\cite[Lemma~1.5]{CD}.

\begin{lemma} \label{lemma: jonsson}
  Let $\class{Q}$ be the quasivariety generated by a class of algebras $\class{K}$. Then each r.f.s.i.\ algebra in $\class{K}$ lies in $\mathbb{ISP}_{\mathrm{U}}(\class{K})$.
\end{lemma}

\begin{theorem}
  In the quasivariety of semi-prime-pointed lattices:
\begin{enumerate}[(i)]
\item The relatively simple (dually) integral algebras are the (dually) integral two-element chains.
\item The r.s.i.\ (r.f.s.i.)\ integral algebras are the integral two-element chains and the algebras of the form ${ \alg{A} \oplus \one }$ where $\alg{A}$ is an s.i.\ (f.s.i.)\ lattice which does not have a join irreducible top element.
\item The r.s.i.\ (r.f.s.i.)\ dually integral algebras are the s.i.\ (f.s.i.)\ dually integral pointed lattices.
\item The r.s.i.\ (r.f.s.i.)\ conic algebras are the r.s.i.\ (r.f.s.i.)\ integral algebras and the r.s.i.\ (r.f.s.i.)\ dually integral algebras.
\end{enumerate}
\end{theorem}

\begin{proof}
  By Lemma~\ref{lemma: jonsson} each r.f.s.i.\ algebra in the quasivariety $\class{Q}$ of semi-prime-pointed lattices is prime-pointed, since $\class{Q}$ is generated by the universal class of prime-pointed lattices. In the following, we only deal with r.s.i.\ algebras. The proofs for r.f.s.i.\ algebras are obtained by obvious modifications of the proofs below.

  Since r.s.i.\ algebras are non-trivial, this means that it has the form $\alg{A} \oplus \one$ for some lattice $\alg{A}$. If $\alg{A}$ is trivial, then $\alg{A} \oplus \one$ is the (relatively simple) two-element chain. Suppose therefore that $\alg{A}$ is non-trivial. Then the equivalence relation with two equivalence classes $A$ and $\{ \one \}$ is a $\class{Q}$-congruence of $\alg{A} \oplus \one$ which is neither the identity nor the total congruence. Consequently, $\alg{A} \oplus \one$ is not simple. Conversely, each two-element chain is relatively simple, proving (i).

  Now suppose moreover that $\alg{A} \oplus \one$ is r.s.i.\ in $\class{Q}$. We show that $\alg{A}$ is subdirectly irreducible. Each lattice embedding $\iota\colon \alg{A} \into \prod_{i \in I} \alg{B}_{i}$ extends to a pointed lattice embedding $\overline{\iota}\colon (\alg{A} \oplus \one) \into \prod_{i \in I} (\alg{B}_{i} \oplus \one)$. Because $\alg{A} \oplus \one$ is r.s.i., there is some $i \in I$ such that the map $\overline{\pi}_{i} \circ \overline{\iota}\colon (\alg{A} \oplus \one) \to (\alg{B}_{i} \oplus \one)$ is an embedding, where $\overline{\pi}_{i}\colon \prod_{i \in I} (\alg{B}_{i} \oplus \one) \to (\alg{B}_{i} \oplus \one)$ is the projection map. Consequently, $\pi_{i} \circ \iota\colon \alg{A} \into \alg{B}_{i}$ is also an embedding.

  Next, suppose that $\alg{A}$ has a join irreducible top element $\top$. Let $\theta$ be the equivalence relation on $\alg{A} \oplus \one$ whose only non-singleton equivalence class is $\{ \top, \one \}$, and let $\psi$ be the equivalence relation on $\alg{A} \oplus \one$ with the two equivalence classes $\{ \one \}$ and $A$. Because $\top$ is join irreducible, $\theta$ is a $\class{Q}$-congruence of $\alg{A} \oplus \one$. So is $\psi$, since $(\alg{A} \oplus \one) / \psi$ is a two-element chain. But $\theta \cap \psi = \idcon_{\alg{A}}$, so $\alg{A}$ fails to be r.s.i. This proves the left-to-right inclusion in (ii).

  Conversely, let $\alg{A}$ be a subdirectly irreducible lattice which does not have a join irreducible top element and let $\iota\colon (\alg{A} \oplus \one) \into \prod_{i \in I} \alg{C}_{i}$ for $i \in I$ be a pointed lattice embedding with $\alg{C}_{i} \in \class{Q}$. We need to find $i \in I$ such that $\pi_{i} \circ \iota\colon (\alg{A} \oplus \one) \to \alg{C}_{i}$ is an embedding. It suffices to deal with the case where the algebras $\alg{C}_{i} \in \class{Q}$ are r.s.i.\ and therefore prime-pointed. Indeed, we may assume that $\alg{C}_{i} = \alg{B}_{i} \oplus \one$ for some lattice $\alg{B}_{i}$. Let $\varepsilon\colon \alg{A} \into \alg{A} \oplus \one$ be the inclusion embedding. Because $\alg{A}$ is subdirectly irreducible, the embedding $\iota \circ \varepsilon\colon \alg{A} \into \prod_{i \in I} (\alg{B}_{i} \oplus \one)$ yields some $i \in I$ such that $\pi_{i} \circ \iota \circ \varepsilon\colon \alg{A} \into (\alg{B}_{i} \oplus \one)$ is an embedding.

  We claim that the homomorphism $\pi_{i} \circ \iota\colon (\alg{A} \oplus \one) \to (\alg{B}_{i} \oplus \one)$ is also an embedding. If not, there is some $a \in \alg{A}$ such that $(\pi_{i} \circ \iota \circ \varepsilon)(a) = \one \in (\alg{B}_{i} \oplus \one)$. Because $\pi_{i} \circ \iota \circ \varepsilon$ is an embedding, there can be at most one such $a \in \alg{A}$. But if $(\pi_{i} \circ \iota \circ \varepsilon)(a) = \one$, then $(\pi_{i} \circ \iota \circ \varepsilon)(b)$ for each $b \geq a$, therefore $a$ must be the largest element of~$\alg{A}$. By assumption, $a$ is not join irreducible in the non-trivial lattice $\alg{A}$, therefore neither is the element $(\pi_{i} \circ \iota \circ \varepsilon)(a) = \one$ in $\alg{B}_{i} \oplus \one$, since $\pi_{i} \circ \iota \circ \varepsilon$ is a lattice embedding. But this is a contradiction: the element $\one$ in fact is join irreducible in $\alg{B}_{i} \oplus \one$. This proves the right-to-left inclusion in (ii).

  (iii): trivial. (iv): this follows from the fact that each semiconic pointed lattice is a subdirect product of an integral and a dually integral one (Fact~\ref{fact: sp of cones}).
\end{proof}

\section{Pointed lattice subreducts of varieties of RLs}
\label{section: rls}

  We shall now describe the pointed lattice subreducts of some varieties of RLs. Given a positive universal class $\class{K}$ such that $\class{K} \vDash \one \wedge (x \vee y) \equals (\one \wedge x) \vee (\one \wedge y)$, our first result (Theorem~\ref{thm: main}) states that the class of pointed lattice subreducts of semi-$\class{K}$ RL and CRLs, i.e.\ RLs and CRLs in the varieties $\mathbb{V}(\RL_{\class{K}})$ and $\mathbb{V}(\CRL_{\class{K}})$ generated by RLs and CRLs whose pointed lattice reduct lies in~$\class{K}$, is the class of semi-$\class{K}_{\pp}$ pointed lattices, i.e.\ the (quasi)variety generated by the prime-pointed lattices in $\class{K}$. This in particular describes the pointed lattice subreducts of integral and of semiconic RLs and CRLs (Corollary~\ref{cor: main}). Our second result (Theorem~\ref{thm: pre-k}) states that the pointed lattice subreducts of left pre-$\class{K}$ RLs (which in general properly contains the variety of semi-$\class{K}$ RLs) coincide with the pointed lattice subreducts of semi-$\class{K}$ pointed lattices. In particular, the pointed lattice subreducts of left prelinear RLs are precisely the distributive pointed lattices.

  A \emph{residuated lattice}, abbreviated as \emph{RL}, is an algebra $\langle A; \wedge, \vee, \cdot, \one, \bs, / \rangle$ such that $\langle A; \wedge, \vee \rangle$ is a lattice, $\langle A; \cdot, \one \rangle$ is a monoid, the multiplication preserves the lattice order~$\leq$ in both co-ordinates, and the left and right division operations $\bs$ and $/$ are linked to multiplication by the \emph{residuation laws}:
\begin{align*}
  x \leq z / y \iff x \cdot y \leq z \iff y \leq x \bs z.
\end{align*}
  Left and right division are thus uniquely determined by the multiplication and the lattice order. Residuated lattices form a variety.

  A \emph{commutative residuated lattice}, abbreviated as \emph{CRL}, is a residuated lattice whose monoidal reduct is commutative, i.e.\ it satisfies the equation $x \cdot y \equals y \cdot x$. The two residuals then coincide and we can use the symbol $\rightarrow$ for both of them.

  We use the following order of precedence: products bind more tightly than divisions, which in turn bind more tightly than lattice operations. For the basic arithmetic laws of residuated lattices, see~\cite{Glimpse,RLBook}.

  The \emph{pointed lattice reduct} of a residuated lattice $\alg{A}$ is the reduct $\langle A; \wedge, \vee, \one \rangle$. Given a positive universal class $\class{K}$ of pointed lattices, we use the following notation:
\begin{itemize}
\item $\RL_{\class{K}}$ is the class of all RLs whose pointed lattice reduct lies in $\class{K}$,
\item $\CRL_{\class{K}}$ is the class of all CRLs whose pointed lattice reduct lies in $\class{K}$.
\end{itemize}
  Clearly $\RL_{\class{K}}$ and $\CRL_{\class{K}}$ are positive universal classes, so J\'{o}nsson's lemma ensures, as it did for pointed lattices, that the variety generated by $\RL_{\class{K}}$ ($\CRL_{\class{K}}$) is the class of all algebras isomorphic to subdirect products of RLs in $\RL_{\class{K}}$ (of CRLs in $\CRL_{\class{K}}$).

\begin{definition}
  The class of \emph{semi-$\class{K}$ RLs (CRLs)} is the variety generated by $\RL_{\class{K}}$ (by $\CRL_{\class{K}}$), or equivalently the class of all algebras isomorphic to a subdirect product of RLs in $\RL_{\class{K}}$ (of CRLs in $\CRL_{\class{K}}$).
\end{definition}

  For example, a semiconic RL is up to isomorphism a subdirect product of RLs with a conic pointed lattice reduct. It is important to keep in mind that being a semi-$\class{K}$ RL or CRL is in general a stronger property than being a RL or a CRL with a semi-$\class{K}$ pointed lattice reduct. For example, being a semiconic or a semilinear RL is a stronger property than having a semiconic or a semilinear (i.e.\ distributive) or a pointed lattice reduct, as we shall see below.

  Semiconic RLs were introduced in~\cite{ConicHR} and a structure theory for semiconic idempotent ($x \cdot x \equals x$) RLs was recently developed in~\cite{ConicFG}. The term \emph{well-connected} is commonly used for RLs whose pointed lattice reduct is prime-pointed. Well-connected RLs arise in connection with the disjunction property~\cite[Theorem~5.2.1]{Glimpse}.  

  We first show, using a non-integral generalization of a construction known as \emph{drastic multiplication}, that each prime-pointed lattice is a subreduct of a CRL.

\begin{lemma} \label{lemma: drastic}
  Let $\alg{A}$ be a bounded pointed lattice with a splitting pair $\pair{\one}{\coone}$, i.e.\ $\one \nleq \coone$ and for each $a \in \alg{A}$ either $\one \leq a$ or $a \leq \coone$. Then $\alg{A}$ is the pointed lattice reduct of a simple CRL with the following operations:
\begin{align*}
  a \cdot b \assign
  \begin{cases}
    \bot & \text{ if } a \ngeq \one \text{ and } b \ngeq \one, \\
    a & \text{ if } a \ngeq \one \text{ and } b \geq \one, \\
    b & \text{ if } a \geq \one \text{ and } b \ngeq \one, \\
    a \vee b & \text{ if } a \geq \one \text{ and } b \geq \one,
  \end{cases}
\end{align*}
  and
\begin{align*}
  a \rightarrow b \assign
  \begin{cases}
    \top & \text{ if } a \ngeq \one \text{ and } a \leq b, \\
    \coone & \text{ if } a \ngeq \one \text{ and } a \nleq b, \\
    b & \text{ if } a \geq \one \text{ and } b \ngeq \one, \\
    b & \text{ if } a \geq \one \text{ and } b \geq \one \text{ and } a \leq b, \\
    b \wedge \coone & \text{ if } a \geq \one \text{ and } b \geq \one \text{ and } a \nleq b.
  \end{cases}
\end{align*}
\end{lemma}

\begin{proof}
  Multiplication is clearly commutative. It suffices to verify the associativity of multiplication and that division is the residual of multiplication. (The monotonicity of multiplication then follows.) This is a straightforward case analysis. This residuated lattice is simple: each non-trivial congruence $\theta$ of a residuated lattice contain a pair of the form $\pair{a}{\one}$ for some $a < \one$, but then $\pair{\bot}{\one} = \pair{a^{2}}{\one^{2}} \in \theta$ and $\pair{\bot}{x} = \pair{\bot \cdot x}{\one \cdot x} \in \theta$ for each $x$, so $\theta$ is the total congruence.
\end{proof}

  The above multiplication is the smallest residuated multiplication operation (or equivalently, the smallest multiplication operation distributing over finite joins) which has $\one$ as a unit. In addition to being simple, the above CRL satisfies a number of non-trivial equations, such as ${x^{2} \equals x^{3}}$ and $\one \wedge (x \cdot y) \equals (\one \wedge x) \cdot (\one \wedge y)$.

\begin{theorem} \label{thm: drastic}
  Each prime-pointed lattice in a positive universal class of pointed lattices $\class{K}$ is a subreduct of a (simple) CRL with a pointed lattice reduct in $\class{K}$.
\end{theorem}

\begin{proof}
  If $\alg{A} \in \class{K}$ is prime-pointed, then its ideal completion $\Idl \alg{A}$ lies in $\class{K}$ by Fact~\ref{fact: idl preserves positive universal sentences} and the constant $\one$ is completely join prime in $\Idl \alg{A}$ by Lemma~\ref{lemma: join prime in idl}. The pair $\pair{\one}{\coone}$ for $\coone \assign \bigvee \set{x \in \alg{A}}{a \nleq x}$ is thus a splitting pair in $\Idl \alg{A}$, which is therefore the pointed lattice reduct of a simple CRL by Lemma~\ref{lemma: drastic}.
\end{proof}

  The following corollary is well known.

\begin{corollary}
  Each lattice is a subreduct of a (simple) integral CRL.
\end{corollary}

\begin{proof}
  Each lattice $\alg{A}$ is a subreduct of the prime-pointed unital lattice $\alg{A} \oplus \one$.
\end{proof}

  We now turn to the converse problem: given a RL or a CRL in a certain variety, what can we say about its pointed lattice reduct? We show that each RL is semi-irreducible-pointed as a pointed lattice. To prove this for CRLs, it is sufficient to observe that every (finitely) subdirectly irreducible CRL is irreducible-pointed. This was proved in~\cite[p.~262]{Glimpse} (for subdirectly irreducible integral CRLs, but the same argument works for arbitrary finitely subdirectly irreducible CRLs). Because every CRL is isomorphic to a subdirect product of (finitely) subdirectly irreducible CRLs, it must therefore be semi-irreducible-pointed.

  The above argument does not go through for general RLs, since subdirectly irreducible RLs need not be irreducible-pointed (see~\cite[Figure~3.7, p.~203]{Glimpse}). This reflects the fact that, beyond the commutative case, it is not sufficient to only consider quotients of RLs with respect to congruences. Rather, we need to consider left and right congruences of RLs.

\begin{definition}
  A \emph{left congruence} of a RL $\alg{A}$ is a lattice congruence $\theta$ such that for all $a, x, y \in \alg{A}$
\begin{align*}
  \pair{x}{y} \in \theta \implies \pair{a \cdot x}{a \cdot y} \in \theta \text{ and } \pair{a \bs x}{a \bs y} \in \theta.
\end{align*}
  A \emph{right congruence} of $\alg{A}$ is a lattice congruence $\theta$ such that for all $a, x, y \in \alg{A}$
\begin{align*}
  \pair{x}{y} \in \theta \implies \pair{x \cdot a}{y \cdot a} \in \theta \text{ and } \pair{x / a}{y / x} \in \theta.
\end{align*}
\end{definition}

  The lattice of all left (right) congruences of a RL $\alg{A}$ will be denoted by $\LCon \alg{A}$ (by $\RCon \alg{A}$) and the lattice of all congruences of $\alg{A}$ by~$\Con \alg{A}$. The lattice of all congruences of the lattice reduct of $\alg{A}$ will be denoted by $\ConLat \alg{A}$.

  We shall use $x \leq_{\theta} y$ for $\theta \in \ConLat \alg{A}$ as an abbreviation for $x / \theta \leq y / \theta$ in $\alg{A} / \theta$:
\begin{align*}
   x \leq_{\theta} y \iff \pair{x \vee y}{y} \in \theta \iff \pair{x}{x \wedge y} \in \theta \iff x / \theta \leq y / \theta \text{ in } \alg{A} / \theta.
\end{align*}
  Observe that for each left congruence $\theta$
\begin{align*}
   x \leq_{\theta} y \implies a \cdot x \leq_{\theta} a \cdot y \text{ and } a \bs x \leq_{\theta} a \bs y.
\end{align*}

\begin{fact} \label{fact: con}
  $\LCon \alg{A}$, $\RCon \alg{A}$, and $\Con \alg{A}$ are complete sublattices of $\ConLat \alg{A}$ for every RL $\alg{A}$. In particular, meets in these lattices are intersections and directed joins are directed unions. They are distributive algebraic lattices whose compact elements are the finitely generated, or equivalently the principal, (left or right) congruences. Moreover, $\Con \alg{A} = \LCon \alg{A} \cap \RCon \alg{A}$ for each RL $\alg{A}$.
\end{fact}

\begin{proof}
  We only prove that (i) if $\theta$ is both a left and a right congruence, then it is a congruence, and (ii) each finitely generated left congruence is principal.

  (i): the only non-trivial part of this claim is that $\pair{x}{y} \in \theta$ implies $\pair{x \bs a}{y \bs a} \in \theta$ and $\pair{a / x}{a / y} \in \theta$. We only prove that if $x \leq_{\theta} y$, then $y \bs a \leq_{\theta} x \bs a$. Indeed, if $x \leq_{\theta} y$, then $\one \leq x \bs x \leq_{\theta} x \bs y$, since $\theta$ is a left congruence. Thus $y \bs a = \one \cdot (y \bs a) \leq_{\theta} (x \bs y) \cdot (y \bs a) \leq x \bs a$, since $\theta$ is a right congruence.

  (ii): let $\LCg^{\alg{A}} X$ be the left congruence of $\alg{A}$ generated by a set $X \subseteq \alg{A}^{2}$. Observe that each principal left congruence of $\alg{A}$ has the form $\LCg^{\alg{A}} \pair{a}{b}$ for some $a \leq b$ in $\alg{A}$, since $\alg{A}$ has a lattice reduct. In that case $\LCg^{\alg{A}} \pair{a}{b} = \LCg^{\alg{A}} \pair{\one}{\one \wedge b \bs a}$: for each left congruence $\theta$ of $\alg{A}$, if $\pair{a}{b} \in \theta$, then $\pair{\one \wedge b \bs a}{\one} = \pair{\one \wedge b \bs a}{\one \wedge b \bs b} \in \theta$, and conversely if $\pair{\one}{\one \wedge b \bs a} \in \theta$, then $b = b \cdot \one \leq_{\theta} b \cdot (\one \wedge b \bs a) \leq b \cdot (b \bs a) \leq a$. Thus each principal left congruence of $\alg{A}$ has the form $\LCg^{\alg{A}} \pair{\one}{c}$ for some $c \leq \alg{A}$. But given $c_{1}, \dots, c_{n} \in \alg{A}_{-}$ we have $\LCg^{\alg{A}} (\pair{\one}{c_{1}}, \dots, \pair{\one}{c_{n}}) = \LCg^{\alg{A}} (\pair{\one}{c_{1} \wedge \dots \wedge c_{n}})$, proving that each finitely generated left congruence of $\alg{A}$ is in fact principal.
\end{proof}

  We shall generally work with left congruences. Analogous results for right congruences will immediately follow because the right congruences of a RL $\alg{A}$ are precisely the left congruences of the RL obtained from $\alg{A}$ by replacing the binary operations $x \cdot y$ and $x \bs y$ and $y / x$ by the operations $y \cdot x$ and $y / x$ and $x \bs y$.

\begin{fact} \label{fact: con isomorphism}
  Consider a RL $\alg{A}$ and $\theta \in \Con \alg{A}$. Then $\LCon \alg{A} / \theta \iso [\theta, \nabla_{\alg{A}}]_{\LCon \alg{A}}$ via the map $\phi \mapsto \phi / \theta$, where $[\theta, \nabla_{\alg{A}}]_{\LCon \alg{A}}$ is the upset generated by $\theta$ in $\LCon \alg{A}$.
\end{fact}

\begin{proof}
  Given a congruence $\theta$ of a RL $\alg{A}$, the map $\phi \mapsto \phi / \theta$ is known to be an isomorphism between $[\theta, \nabla_{\alg{A}}]_{\ConLat \alg{A}}$ and $\ConLat \alg{A} / \theta$. It suffices to show that $\phi$ is a left congruence of $\alg{A}$ if and only if $\phi / \theta$ is a left congruence of $\alg{A}$. If $\phi$ is a left congruence, then given $\pair{x / \theta}{y / \theta} \in \phi / \theta$ and $a / \theta \in \alg{A} / \theta$ we have $\pair{x}{y} \in \phi$, so $\pair{a \cdot x}{a \cdot y} \in \phi$ and $\pair{(a / \theta) \cdot (x / \theta)}{(a / \theta) \cdot (y / \theta)} \in \phi / \theta$, and likewise for left division. Conversely, given an equivalence relation $\phi$ on $\alg{A}$ containing $\theta$ such that $\phi / \theta$ is a left congruence of~$\alg{A} / \theta$, we need to show that $\phi$ is a left congruence of $\alg{A}$. Indeed, if $\pair{x}{y} \in \phi$, then $\pair{x / \theta}{y / \theta} \in \phi / \theta$, so $\pair{(a \cdot x) / \theta}{(a \cdot y) / \theta} \in \phi / \theta$ and $\pair{a \cdot x}{a \cdot y} \in \phi$, and likewise for left division.
\end{proof}

  The congruences of a RL $\alg{A}$ can be represented by subsets of $\alg{A}$~\cite[Theorem~3.47]{Glimpse}. The same holds for the left and right congruences. While congruences can be represented in a number of equivalent ways (by normal multiplicative $\one$-filters, convex normal subuniverses, or convex normal submonoids), for our purposes the simplest representation of left and right congruences is by multiplicative $\one$-filters.

\begin{definition}
  A $\one$-filter $F$ on a RL $\alg{A}$ is \emph{multiplicative} if
\begin{align*}
  x, y \in F \implies x \cdot y \in F.
\end{align*}
  A $\one$-filter on a residuated lattice is \emph{normal} if for all $x, y \in \alg{A}$
\begin{align*}
  x \bs y \in F \iff y / x \in F.
\end{align*}
\end{definition}

  The lattice of all multiplicative $\one$-filters of $\alg{A}$ ordered by inclusion will be denoted by $\Fi \alg{A}$, the lattice of all normal multiplicative $\one$-filters by $\NFi \alg{A}$. The multiplicative $\one$-filter generated by a set $X \subseteq \alg{A}$ will be denoted by $\Fg^{\alg{A}} X$.

  Each finitely generated multiplicative $\one$-filter is in fact principal, since
\begin{align*}
  \Fg^{\alg{A}} (a_{1}, \dots, a_{n}) = \Fg^{\alg{A}}(a_{1} \wedge \dots \wedge a_{n}).
\end{align*}
  Observe also that $\Fg^{\alg{A}}(a) = \Fg^{\alg{A}}(\one \wedge a)$.

  Each multiplicative $\one$-filter $F$ of a RL $\alg{A}$ determines the following relations $\LTheta(F)$ and $\RTheta(F)$, which we shall presently prove to be a left and a right congruence of~$\alg{A}$:
\begin{align*}
  \pair{x}{y} \in \LTheta(F) & \iff x \bs y, y \bs x \in F \iff x \cdot f \leq y \text{ and } y \cdot f \leq y \text{ for some } f \in F, \\
  \pair{x}{y} \in \RTheta(F) & \iff x / y, y / x \in F \iff f \cdot x \leq y \text{ and } f \cdot y \leq x \text{ for some } f \in F.
\end{align*}
  If $\alg{A}$ is a RL and $\LTheta(F) = \RTheta(F)$, we write $\CTheta(F) \assign \LTheta(F) = \RTheta(F)$.

  Conversely, each left or right congruence $\theta$ of a RL $\alg{A}$ determines the following set, which we shall presently prove to be a multiplicative $\one$-filter:
\begin{align*}
  \up_{\theta} \one \assign \set{a \in \alg{A}}{\one \leq_{\theta} a}.
\end{align*}

\begin{fact}
  The following are equivalent for each multiplicative $\one$-filter $F$ of a RL~$\alg{A}$:
\begin{enumerate}[(i)]
\item $F$ is normal.
\item $\LTheta(F) = \RTheta(F)$.
\item If $f \in F$, then $a \bs f a \in F$ and $a f / a \in F$ for each $a \in \alg{A}$.
\end{enumerate}
\end{fact}

\begin{proof}
  (i) $\Rightarrow$ (ii): trivial. (ii) $\Rightarrow$ (i): if $\LTheta(F) = \RTheta(F)$, then this equivalence relation $\theta$ is in fact a congruence, so
\begin{align*}
  x \bs y \in F & \iff \one / \theta \leq (x / \theta) \bs (y / \theta) \text{ in } \alg{A} / \theta \\
  & \iff x / \theta \leq y / \theta  \text{ in } \alg{A} / \theta \\
  & \iff \one / \theta \leq (y / \theta) / (x / \theta)  \text{ in } \alg{A} / \theta \\
  & \iff y / x \in F.
\end{align*}

  (i) $\Rightarrow$ (iii): if $f \in F$, then $a \bs a f \in F$ and $f a / a \in F$ because $f \leq a \bs a f$ and $f \leq f a / a$, so $a f / a, a \bs f a \in F$ by normality.

  (iii) $\Rightarrow$ (i): if $x \bs y \in F$, then $(x \cdot (x \bs y)) / x \in F$, but $(x \cdot (x \bs y)) / x \leq y / x$, so $y / x \in F$. Likewise, if $y / x \in F$, then $x \bs ((y / x) \cdot x) \in F$, but $x \bs ((y / x) \cdot x) \leq x \bs y$.
\end{proof}

\begin{theorem} \label{thm: theta iso}
  Let $\alg{A}$ be a RL. Then:
\begin{enumerate}[(i)]
\item $\LTheta\colon \Fi \alg{A} \to \LCon \alg{A}$ is an isomorphism with inverse $\theta \mapsto \up_{\theta} \one$.
\item $\RTheta\colon \Fi \alg{A} \to \RCon \alg{A}$ is an isomorphism with inverse $\theta \mapsto \up_{\theta} \one$.
\item $\CTheta\colon \NFi \alg{A} \to \Con \alg{A}$ is an isomorphism with inverse $\theta \mapsto \up_{\theta} \one$.
\end{enumerate}
\end{theorem}

\begin{proof}
  (ii): follows from (i) by the remark following Fact~\ref{fact: con}.

  (iii): this claim was proved in~\cite[Theorem~3.47]{Glimpse}. Alternatively, it follows from (i): if $F$ is normal, then $\LTheta(F) = \RTheta(F) \in \LCon \alg{A} \cap \RCon \alg{A} = \Con \alg{A}$, and conversely if $\theta$ is a congruence, then $\up_{\theta} \one$ is normal, since
\begin{align*}
  \one \leq_{\theta} x \bs y \iff x \leq_{\theta} y \iff \one \leq_{\theta} y / x.
\end{align*}

  (i): consider a left congruence $\theta$. Clearly $\up_{\theta} \one$ is a $\one$-filter. It is a multiplicative one, i.e.\ $\one \leq_{\theta} a, b$ implies $\one \leq_{\theta} a \cdot b$, because $\one \leq_{\theta} b$ implies $a = a \cdot \one \leq_{\theta} a \cdot b$, and then $\one \leq_{\theta} a$ implies $\one \leq_{\theta} a \leq_{\theta} a \cdot b$.

  Conversely, consider a multiplicative $\one$-filter $F$ on $\alg{A}$. Then the relation $\LTheta(F)$ is reflexive because $x \bs x \geq \one \in F$, symmetric by definition, and transitive because $(x \bs y) \cdot (y \bs z) \leq x \bs z$. It is a left congruence due to the following implications:
\begin{align*}
  x \cdot f \leq x' ~ \& ~ y \cdot f \leq y' & \implies (x \wedge y) \cdot f \leq x' \wedge y' \text{ and } (x \vee y) \cdot f \leq x' \vee y', \\
  x \cdot f \leq x' & \implies (a \cdot x) \cdot f \leq a \cdot x' \text{ and } (a \bs x) \cdot f \leq a \bs x'.
\end{align*}
  If $\theta$ is a left congruence, then $\up_{\theta} \one$ is clearly a $\one$-filter. Moreover, it is a multi\-plicative one: if $x, y \in \up_{\theta} \one$, then $\one \leq_{\theta} x$ and $\one \leq_{\theta} y$, so $\one \leq_{\theta} x = x \cdot \one \leq_{\theta} x \cdot y$ and $x \cdot y \in \up_{\theta} \one$.

  The two maps are clearly order-preserving. It remains to prove that they are mutually inverse bijections:
\begin{align*}
  \pair{x}{y} \in \LTheta(\up_{\theta} \one) & \iff x \cdot f \leq y \text{ and } y \cdot f \leq x \text{ for some } f \in \up_{\theta} \one \\
  & \iff x \leq_{\theta} y \text{ and } y \leq_{\theta} x \\
  & \iff \pair{x}{y} \in \theta.
\end{align*}
  and
\begin{align*}
  a \in \up_{\LTheta(F)} \one & \iff \one \leq_{\LTheta(F)} a \\
  & \iff \one \cdot f \leq a \text{ for some } f \in F \\
  & \iff a \in F. \qedhere
\end{align*}
\end{proof}

\begin{corollary} \label{cor: lcon rcon}
  $\LCon \alg{A} \iso \RCon \alg{A}$ for each RL $\alg{A}$ via an isomorphism which restricts to the identity on $\Con \alg{A}$.
\end{corollary}

\begin{corollary} \label{cor: fi a algebraic}
  $\Fi \alg{A}$ is a distributive algebraic lattice for each RL $\alg{A}$.
\end{corollary}

  The following two lemmas are known facts about multiplicative $\one$-filters.

\begin{lemma} \label{lemma: fg}
  Let $\alg{A}$ be a RL. Then for all $a, b \in \alg{A}$
\begin{align*}
  b \in \Fg^{\alg{A}}(a) \iff (\one \wedge a)^{k} \leq b \text{ for some } k \in \omega.
\end{align*}
\end{lemma}

\begin{lemma} \label{lemma: fg a vee b}
  Let $\alg{A}$ be a RL. Then for all $a, b \in \alg{A}$
\begin{align*}
  \Fg^{\alg{A}}(a) \cap \Fg^{\alg{A}}(b) & = \Fg^{\alg{A}}((\one \wedge a) \vee (\one \wedge b)), \\ \Fg^{\alg{A}}(a) \vee \Fg^{\alg{A}}(b) & = \Fg^{\alg{A}}(a, b) = \Fg^{\alg{A}}(a \wedge b).
\end{align*}
\end{lemma}

\begin{proof}
  The second equality is immediate, and
\begin{align*}
  \Fg^{\alg{A}}((\one \wedge a) \vee (\one \wedge b)) \subseteq \Fg^{\alg{A}}(\one \wedge a) \cap \Fg^{\alg{A}}(\one \wedge b) = \Fg(a) \cap \Fg(b).
\end{align*}
  Conversely, take $c \in \Fg(a) \cap \Fg(b)$. By Lemma~\ref{lemma: fg} there are $m, n \in \omega$ such that $(\one \wedge a)^{m} \leq c$ and $(\one \wedge a)^{n} \leq c$. Taking $k \assign \max(m, n)$, we have $(\one \wedge a)^{k} \vee (\one \wedge b)^{k} \leq c$. If we distribute products over joins, $((\one \wedge a) \vee (\one \wedge b))^{2k}$ is a join of products of $2k$ factors where $i \in \{ 0, \dots, 2k \}$ of the factors are equal to $\one \wedge a$ and $2k-i$ of the factors are equal to $\one \wedge b$. Each such product lies either below $(\one \wedge a)^{k}$ or below $(\one \wedge b)^{k}$. Thus $((\one \wedge a) \vee (\one \wedge b))^{2k}  \leq (\one \wedge a)^{k} \vee (\one \wedge b)^{k} \leq c$ and $c \in \Fg((\one \wedge a) \vee (\one \wedge b))$.
\end{proof}

\begin{corollary}
  The compact elements of $\Fi \alg{A}$ form a sublattice of $\Fi \alg{A}$. Consequently, the same holds for $\LCon \alg{A}$ and $\RCon \alg{A}$.
\end{corollary}

\begin{lemma} \label{lemma: theta f irreducible}
  Let $\alg{A}$ be a RL and $F$ be a meet irreducible element of $\Fi \alg{A}$. Then $\alg{A} / \LTheta(F)$ is an irreducible-pointed lattice.
\end{lemma}

\begin{proof}
  Let $\theta \assign \LTheta(F)$ and consider $a / \theta, b / \theta \leq \one / \theta$ in $\alg{A} / \theta$. We may choose $a, b \leq \one$, since $x / \theta \leq \one / \theta$ implies that $x \cdot f \leq \one$ for some $f \in F$, in which case $x / \theta = y / \theta$ and $y \leq \one$ for $y \assign x \cdot (\one \wedge f)$. Observe that
\begin{align*}
  a \in F \iff a / \theta = \one / \theta.
\end{align*}
  Because $\Fi \alg{A}$ is distributive by Corollary~\ref{cor: fi a algebraic}, $F$ is a meet prime element of $\Fi \alg{A}$. Therefore, using Lemma~\ref{lemma: fg a vee b} in the third equivalence,
\begin{align*}
  (a / \theta) \vee (b / \theta) = \one / \theta & \iff (a \vee b) / \theta = \one / \theta \\
  & \iff a \vee b \in F \\
  & \iff \Fg^{\alg{A}}(a) \cap \Fg^{\alg{A}}(b) = \Fg^{\alg{A}}(a \vee b) \subseteq F \\
  & \iff \Fg^{\alg{A}}(a) \subseteq F \text{ or } \Fg^{\alg{A}}(b) \subseteq F \\
  & \iff a \in F \text{ or } b \in F \\
  & \iff a / \theta = \one \text{ or } b / \theta = \one. \qedhere
\end{align*}
\end{proof}

\begin{fact}
  Each RL is semi-irreducible-pointed as a pointed lattice. Consequently, each RL which satisfies $\one \wedge (x \vee y) \equals (\one \wedge x) \vee (\one \wedge y)$ is semi-prime-pointed.
\end{fact}

\begin{proof}
  Because the lattice $\Fi \alg{A}$ is algebraic, the multiplicative $\one$-filter $\up \one$ is an intersection of meet irreducible elements of $\Fi \alg{A}$. By Theorem~\ref{thm: theta iso}, taking into account the fact that each left congruence is in particular a lattice congruence, the pointed lattice reduct of $\alg{A}$ is isomorphic to a subdirect product of pointed lattices $\alg{A} / \LTheta(F)$ for $F$ meet irreducible in $\Fi \alg{A}$. The first claim now follows from Lemma~\ref{lemma: theta f irreducible}. The second claim follows from Fact~\ref{fact: semi-prime-pointed}.
\end{proof}

  The following definitions can equivalently be stated in terms of either left or congruences by Corollary~\ref{cor: lcon rcon}.

\begin{definition}
  A RL $\alg{A}$ is \emph{strongly s.i.}\ if for each family of left (or equivalently, right) congruences $(\theta_{i})_{i \in I}$ of~$\alg{A}$, if $\bigcap_{i \in I} \theta_{i} = \idcon_{\alg{A}}$, then $\theta_{i} = \idcon_{\alg{A}}$ for some $i \in I$. It is \emph{strongly f.s.i.}\ if this implication holds for $I$ finite. It is \emph{strongly simple} if $\LCon \alg{A}$ (or equivalently, $\RCon \alg{A}$) is a two-element lattice.
\end{definition}

\begin{definition}
  A RL $\alg{A}$ is \emph{semisimple} if it is isomorphic to a subdirect product of simple RLs, or equivalently if $\idcon_{\alg{A}}$ is an intersection of maximal non-trivial congruences. It is \emph{strongly semisimple} if $\idcon_{\alg{A}}$ it is isomorphic to a subdirect product of strongly simple RLs, or equivalently is an intersection of congruences which are maximal non-trivial left (or equivalently, right) congruences.
\end{definition}

  Because in CRLs every left or right congruence is a congruence, strongly simple (semisimple, f.s.i., s.i.) CRLs coincide with simple (semisimple, f.s.i., s.i.) CRLs.

  To make the following theorem more concise, in its statement we shall count the trivial RL as (strongly) simple and therefore as (strongly) s.i.\ and f.s.i..

\begin{theorem} \label{thm: main}
  Let $\class{K}$ be a positive universal class of pointed lattices such that $\class{K} \vDash \one \wedge (x \vee y) \equals (\one \wedge x) \vee (\one \wedge y)$. Then:
\begin{enumerate}[(i)]
\item The class of pointed lattice subreducts of f.s.i.\ semi-$\class{K}$ CRLs is $\class{K}_{\pp}$.
\item The class of pointed lattice subreducts of semi-$\class{K}$ CRLs is $\mathbb{ISP}(\class{K}_{\pp})$.
\end{enumerate}
  Moreover:
\begin{enumerate}[(i)]
\setcounter{enumi}{2}
\item The class of pointed lattice subreducts of strongly f.s.i.\ semi-$\class{K}$ RLs is $\class{K}_{\pp}$.
\item The class of pointed lattice subreducts of semi-$\class{K}$ RLs is $\mathbb{ISP}(\class{K}_{\pp})$.
\end{enumerate}
  We can replace \emph{(strongly) f.s.i.}\ by \emph{(strongly) s.i.}\ or by \emph{(strongly) simple} above. We can also replace \emph{CRLs} by \emph{semisimple CRLs} and \emph{RLs} by \emph{strongly semisimple RLs}.
\end{theorem}

\begin{proof}
  Claim (i) is proved by replacing \emph{RLs} by \emph{CRLs} throughout the proof of claim (iii). While (iv) is not an immediate consequence of (iii), (ii) does follow immediately from (i), since every semi-$\class{K}$ CRL is up to isomorphism a subdirect product of f.s.i.\ semi-$\class{K}$ CRLs, which are in fact f.s.i.\ CRLs in $\CRL_{\class{K}}$.

  (iii): each prime-pointed lattice in $\class{K}$ is a subreduct of a simple (and therefore strongly simple) CRL by Theorem~\ref{thm: drastic}. Conversely, let $\alg{A}$ be a strongly f.s.i.\ semi-$\class{K}$ RL. Because $\alg{A}$ is f.s.i.\ and the variety of semi-$\class{K}$ RLs is generated as a quasivariety by $\RL_{\class{K}}$, J\'{o}nsson's lemma for quasivarieties (Lemma~\ref{lemma: jonsson}) implies that $\alg{A} \in \mathbb{ISP}_{\mathrm{U}}(\RL_{\class{K}})$. Because $\class{K}$ is a universal class, it follows that the pointed lattice reduct of $\alg{A}$ lies in $\class{K}$. Moreover, since $\alg{A}$ is strongly f.s.i., $\up \one$ is meet irreducible in $\Fi \alg{A}$ by Theorem~\ref{thm: theta iso}, so $\alg{A}$ is irreducible-pointed by Lemma~\ref{lemma: theta f irreducible}. Because $\class{K} \vDash \one \wedge (x \vee y) \equals (\one \wedge x) \vee (\one \wedge y)$, it is prime-pointed by Fact~\ref{fact: semi-prime-pointed}.

  (iv): the right-to-left inclusion follows immediately from (iii), since the class of semi-$\class{K}$ RLs is closed under $\mathbb{ISP}$. Conversely, because each semi-$\class{K}$ RL is up to isomorphism a subdirect product of RLs in $\RL_{\class{K}}$, it suffices to show that each $\alg{A} \in \RL_{\class{K}}$ lies in $\mathbb{ISP}(\class{K}_{\pp})$. Because $\Fi \alg{A}$ is an algebraic lattice by Corollary~\ref{cor: fi a algebraic}, $\up \one$ is an intersection of meet irreducible elements of $\Fi \alg{A}$, so the pointed lattice reduct of $\alg{A}$ is up to isomorphism a subdirect product of the pointed lattices $\alg{A} / \LTheta(F)$ where $F$ ranges over the meet irreducibles in $\Fi \alg{A}$. But $\alg{A} / \LTheta(F)$ is prime-pointed for each meet irreducible $F \in \Fi \alg{A}$ by Lemma~\ref{lemma: theta f irreducible}, and $\alg{A} / \LTheta(F) \in \class{K}$ because $\class{K}$ is a positive universal class and $\alg{A} \in \class{K}$. Thus $\alg{A} / \LTheta(F) \in \class{K}_{\pp}$ for each meet irreducible $F \in \Fi \alg{A}$.
\end{proof}

  In all claims in Theorem~\ref{thm: main} other than (iv), instead of assuming that $\class{K}$ is a positive universal class, the proof works equally well if we only assume that $\class{K}$ is a universal class such that each prime-pointed lattice in $\class{K}$ embeds into a bounded pointed lattice in $\class{K}$ with a splitting pair $\pair{\one}{\coone}$.

\begin{corollary} \label{cor: main}
  The pointed lattice subreducts of
\begin{enumerate}[(i)]
\item integral (C)RLs are precisely the integral semi-prime-pointed lattices,
\item semiconic (C)RLs are precisely the semiconic semi-prime-pointed lattices,
\item semilinear (C)RLs are precisely the distributive semi-prime-pointed \mbox{lattices}.
\end{enumerate}
\end{corollary}

\begin{proof}
  This follows immediately from Theorem~\ref{thm: main} and Fact~\ref{fact: semi-kpp}.
\end{proof}

  In the remainder of this section, we introduce the varieties of left pre-$\class{K}$ RLs and CRLs and show that their pointed lattice subreducts coincide with the pointed lattice subreducts of semi-$\class{K}$ RLs and CRLs (Theorem~\ref{thm: pre-k subreducts}).

\begin{definition}
  Let $\class{K}$ be a positive universal class of pointed lattices and let $\alg{A}$ be a RL. Then a multiplicative $\one$-filter $F$ of $\alg{A}$ is a \emph{left $\class{K}$-filter} if $\alg{A} / \LTheta(F) \in \class{K}$.
\end{definition}

\begin{definition} \label{def: pre-k}
  Given a positive universal sentence
\begin{align*}
  \Phi \assign (t_{1} \leq u_{1} \mathrm{~or~} \dots \mathrm{~or~} t_{k} \leq u_{k})
\end{align*}
  in the signature of pointed lattices, let $\Pre(\Phi)$ be the equation
\begin{align*}
  \Pre(\Phi) \assign (\one \wedge t_{1} \bs u_{1}) \vee \dots \vee (\one \wedge t_{k} \bs u_{k}) \equals \one
\end{align*}
  in the signature of RLs. Given a positive universal class $\class{K}$ of pointed lattices axiomatized by the positive universal sentences $\Phi_{i}$ for $i \in I$, the variety of \emph{left pre-$\class{K}$} RLs is the variety axiomatized by the equations $\Pre(\Phi_{i})$ for $i \in I$.
\end{definition}

  The fact that the above definition of left pre-$\class{K}$ RLs is independent of the chosen axiomatization of $\class{K}$ will follow immediately from Theorem~\ref{thm: pre-k}, since the definition of a left $\class{K}$-filter does not depend on the choice of axiomatization.

  For example, a RL is \emph{left prelinear} if it satisfies the equation
\begin{align*}
  (\one \wedge x \bs y) \vee (\one \wedge y \bs x) \equals \one,
\end{align*}
  and it is \emph{left preconic} if it satisfies the equation
\begin{align*}
  (\one \wedge x \bs \one) \vee (\one \wedge x) \equals \one.
\end{align*}

\begin{lemma} \label{lemma: pre-k}
  Let $\alg{A}$ be RL. If the pointed lattice reduct of $\alg{A}$ lies in $\class{K}$, then $\alg{A}$ is a left pre-$\class{K}$ RL. The converse implication holds if $\one$ is join irreducible in $\alg{A}$.
\end{lemma}

\begin{proof}
  If $\alg{A}$ lies in $\class{K}$ as a pointed lattice, then for each positive universal sentence $\Phi$ valid in $\class{K}$ and all elements $a_{1}, \dots, a_{n} \in \alg{A}$ there is some $i \in \{ 1, \dots, k \}$ such that $t_{i}^{\alg{A}}(a_{1}, \dots, a_{n}) \leq u_{i}^{\alg{A}}(a_{1}, \dots, a_{n})$. Therefore $\one \wedge t_{i}^{\alg{A}}(a_{1}, \dots, a_{n}) \bs u_{i}^{\alg{A}}(a_{1}, \dots, a_{n}) = \one$ and $\alg{A} \vDash \Pre(\Phi)$.

  Conversely, suppose that $\alg{A}$ is left pre-$\class{K}$ and $\one$ is join irreducible in $\alg{A}$. For each positive universal sentence $\Phi$ in a given axiomatization of $\class{K}$ and for all $a_{1}, \dots, a_{n}$ there is some $i \in \{ 1, \dots, k \}$ such that $\one \wedge t_{i}^{\alg{A}}(a_{1}, \dots, a_{n}) \bs u_{i}^{\alg{A}}(a_{1}, \dots, a_{n}) = \one$, i.e.\ such that $t_{i}^{\alg{A}}(a_{1}, \dots, a_{n}) \leq u_{i}^{\alg{A}}(a_{1}, \dots, a_{n})$. But this means that $\Phi$ holds in the pointed lattice reduct of $\alg{A}$.
\end{proof}

\begin{fact}
  Let $\class{K}$ be a positive universal class of pointed lattices. Then:
\begin{enumerate}[(i)]
\item Each semi-$\class{K}$ RL is left pre-$\class{K}$.
\item Semi-$\class{K}$ CRLs coincide with left pre-$\class{K}$ CRLs.
\end{enumerate}
\end{fact}

\begin{proof}
  This follows immediately from Lemma~\ref{lemma: pre-k}, using the fact that in an s.i.\ CRL the element $\one$ is join irreducible by Theorem~\ref{thm: main} and that each left pre-$\class{K}$ CRL is isomorphic to a subdirect product of s.i.\ left pre-$\class{K}$ CRLs.
\end{proof}

  For example, each semilinear RL is left prelinear and each semiconic RL is left preconic, and the converse implications hold for CRLs (see~\cite[Corolary~5.1.3]{RLBook} and \cite[Section~2.2]{ConicFG}). The lattice-ordered group of order automorphisms of the real line serves as a well-known example of a left prelinear RL (with residuals $x \bs y \assign x^{-1} y$ and $x / y \assign x y^{-1}$) which is not semilinear. It also separates the classes of semiconic and left preconic RLs: each conic and therefore each semiconic RL satisfies
\begin{align*}
  (\one \wedge y \bs x y) \vee (\one \wedge x \bs \one) \equals \one,
\end{align*}
  but this equation fails in this lattice-ordered group if we take $x$ and $y$ to be, respectively, the functions $f, g\colon \mathbb{R} \to \mathbb{R}$ such that $f\colon x \mapsto 2x$ and $g\colon x \mapsto x - 2$.

\begin{theorem} \label{thm: pre-k}
  Let $\class{K}$ be a positive universal class of pointed lattices. Then the following are equivalent for each RL $\alg{A}$:
\begin{enumerate}[(i)]
\item $\alg{A}$ is left pre-$\class{K}$.
\item Each (completely) meet irreducible $F \in \Fi \alg{A}$ is a left $\class{K}$-filter.
\item $\up \one$ is an intersection of left $\class{K}$-filters of $\alg{A}$.
\item Every $F \in \Fi \alg{A}$ is an intersection of left $\class{K}$-filters of $\alg{A}$.
\end{enumerate}
\end{theorem}

\begin{proof}
  (i) $\Rightarrow$ (ii): consider a meet irreducible $F \in \Fi \alg{A}$. The element $\one / \LTheta(F)$ is join irreducible in $\alg{A} / \LTheta(F)$ by Lemma~\ref{lemma: theta f irreducible}. Since $\alg{A}$ is left pre-$\class{K}$, for each positive universal sentence
\begin{align*}
  \Phi \assign (t_{1}(x_{1}, \dots, x_{n}) \leq u_{1}(x_{1}, \dots, x_{n}) \text{ or } \dots \text{ or } t_{k}(x_{1}, \dots, x_{n}) \leq u_{k}(x_{1}, \dots, x_{n}))
\end{align*}
  in a chosen axiomatization of the positive universal class $\class{K}$ and for all elements $a_{1}, \dots, a_{n} \in \alg{A}$ there is some $i \in \{ 1, \dots, k \}$ such that for $\theta \assign \LTheta(F)$
\begin{align*}
  (\one \wedge t_{i}^{\alg{A}}(a_{1}, \dots, a_{n}) \bs u_{i}^{\alg{A}}(a_{1}, \dots, a_{n})) / \theta = \one / \theta.
\end{align*}
  That is, $t_{i}^{\alg{A}}(a_{1}, \dots, a_{n}) \bs u_{i}^{\alg{A}}(a_{1}, \dots, a_{n}) \in F$, so the definition of $\LTheta(F)$ yields that $t_{i}^{\alg{A} / \theta}(a_{1} / \theta, \dots, a_{n} / \theta) \leq u_{i}^{\alg{A} / \theta}(a_{1} / \theta, \dots, a_{n} / \theta)$. Consequently, the positive universal sentence $\Phi$ holds in $\alg{A} / \theta$.

  (ii) $\Rightarrow$ (iv) because $\Fi \alg{A}$ is an algebraic lattice, each $F \in \Fi \alg{A}$ is a meet of completely meet irreducibles in $\Fi \alg{A}$.

  (iv) $\Rightarrow$ (iii): trivial. (iii) $\Rightarrow$ (i): it suffices to show that for each positive universal sentence
\begin{align*}
  \Phi \assign (t_{1}(x_{1}, \dots, x_{n}) \leq u_{1}(x_{1}, \dots, x_{n}) \text{ or } \dots \text{ or } t_{k}(x_{1}, \dots, x_{n}) \leq u_{k}(x_{1}, \dots, x_{n}))
\end{align*}
  valid in $\class{K}$ and all $a_{1}, \dots, a_{n} \in \alg{A}$, the element
\begin{align*}
  (\one \wedge t_{1}^{\alg{A}}(a_{1}, \dots, a_{n}) \bs u_{1}^{\alg{A}}(a_{1}, \dots, a_{n})) \vee \dots \vee (\one \wedge t_{k}^{\alg{A}}(a_{1}, \dots, a_{n}) \bs u_{k}^{\alg{A}}(a_{1}, \dots, a_{n}))
\end{align*}
  lies in $F$ for each left $\class{K}$-filter $F$ of $\alg{A}$. Thus consider a left $\class{K}$-filter $F$ of $\alg{A}$ and let $\theta \assign \LTheta(F)$. Taking $p_{i} \assign t_{i}^{\alg{A} / \theta}(a_{1} / \theta, \dots, a_{m} / \theta)$ and $q_{i} \assign u_{i}^{\alg{A} / \theta}(a_{1} / \theta, \dots, a_{m} / \theta)$, it suffices to show that for each left $\class{K}$-filter $F$ there is some $i \in \{ 1, \dots, k \}$ such that $p_{i} \bs q_{i} \in F$. Because $F$ is a left $\class{K}$-filter, $\alg{A} / \theta \vDash \Phi$, so $p_{i} \leq q_{i}$ for some $i \in \{ 1, \dots, k \}$. Because $\theta$ is a left congruence, $\one \leq p_{i} \bs p_{i} \leq_{\theta} p_{i} \bs q_{i}$, so indeed $p_{i} \bs q_{i} \in \up_{\theta} \one = F$.
\end{proof}

  We remark that left congruences can equivalently be understood as congruences of modules: each RL $\alg{A}$ acts on its pointed lattice reduct by left multiplication and left division. This turns the pointed lattice reduct of $\alg{A}$ into a \emph{left residuated $\alg{A}$-module} $\alg{M}$, and the left congruences of $\alg{A}$ are by definition precisely the congruences of this left residuated $\alg{A}$-module $\alg{M}$, i.e.\ the lattice congruences of $\alg{M}$ compatible with the action of $\alg{A}$. We shall not pursue this angle in more detail here, but it is worth remarking that condition (iii) in Theorem~\ref{thm: pre-k} can equivalently be stated in terms of modules as: the left residuated $\alg{A}$-module $\alg{M}$ is up to isomorphism a subdirect product of left residuated $\alg{A}$-modules with a pointed lattice reduct in~$\class{K}$. (This is similar in spirit but not identical in detail to the semimodules studied in~\cite{GalatosHorcik}, where $\alg{A}$ was an idempotent semiring rather than a residuated lattice and $\alg{A}$ acted on $\alg{M}$ only by multiplication rather than by multiplication and division.)

\begin{theorem} \label{thm: pre-k subreducts}
  Let $\class{K}$ be a positive universal class of pointed lattices such that $\class{K} \vDash \one \wedge (x \vee y) \equals (\one \wedge x) \vee (\one \wedge y)$. Then:
\begin{enumerate}[(i)]
\item The class of pointed lattice subreducts of strongly f.s.i.\ left pre-$\class{K}$ RLs is $\class{K}_{\pp}$.
\item The class of pointed lattice subreducts of left pre-$\class{K}$ RLs is $\mathbb{ISP}(\class{K}_{\pp})$.
\end{enumerate}
  We can replace \emph{strongly f.s.i.}\ by \emph{strongly s.i.}\ or by \emph{strongly simple} above. We can also replace \emph{RLs} by \emph{strongly semisimple RLs}.
\end{theorem}

\begin{proof}
  (i): each algebra in $\class{K}_{\pp}$ is a subreduct of a simple (and therefore strongly simple) semi-$\class{K}$ CRL by Theorem~\ref{thm: main}, which is a left pre-$\class{K}$ RL by Theorem~\ref{thm: pre-k}. Conversely, in each strongly f.s.i.\ left pre-$\class{K}$ RL $\alg{A}$ the element $\one$ is join irreducible by Theorem~\ref{thm: main}, so the pointed lattice reduct of $\alg{A}$ lies in $\class{K}$ by Lemma~\ref{lemma: pre-k}.

  (ii): each algebra in $\mathbb{ISP}(\class{K}_{\pp})$ is a subreduct of a strongly semisimple semi-$\class{K}$ RL by Theorem~\ref{thm: main}, which is a left pre-$\class{K}$ RL by Theorem~\ref{thm: pre-k}. Conversely, let $\alg{A}$ be a left pre-$\class{K}$ RL. By Theorem~\ref{thm: pre-k}, $\up \one$ is an intersection of meet irreducible left $\class{K}$-filters $F \in \Fi \alg{A}$. The pointed lattice reduct of $\alg{A}$ is therefore up to isomorphism a subdirect product of the pointed lattices $\alg{A} / \LTheta(F)$ where $F$ ranges over the meet irreducible left $\class{K}$-filters in $\Fi \alg{A}$. But $\alg{A} / \LTheta(F) \in \class{K}$ by definition and $\alg{A} / \LTheta(F)$ is prime-pointed by Lemma~\ref{lemma: theta f irreducible}. Thus $\alg{A} / \LTheta(F) \in \class{K}_{\pp}$ for each meet irreducible left $\class{K}$-filter $F \in \Fi \alg{A}$.
\end{proof}

  In case $\class{K} \vDash \one \wedge (x \vee y) \equals (\one \wedge x) \vee (\one \wedge y)$, one can replace the positive universal sentences $\Pre(\Phi)$ in the axiomatization of left pre-$\class{K}$ RLs by the sentences
\begin{align*}
  t_{1} \bs u_{1} \vee \dots \vee t_{n} \bs u_{n} \geq \one,
\end{align*}
  provided that one also adds the equation $\one \wedge (x \vee y) \equals (\one \wedge x) \vee (\one \wedge y)$ as an axiom. Clearly this equation together with the displayed inequality implies $\Pre(\Phi)$, and conversely $\Pre(\Phi)$ implies every equation valid in $\class{K}$, so in particular the equation $\one \wedge (x \vee y) \equals (\one \wedge x) \vee (\one \wedge y)$, by Theorem~\ref{thm: pre-k}.

  This immediately yields the following fact, which improves on~\cite[Proposition~6.10]{RLpaper} and~\cite[Corollary~4.2.6]{RLBook} by showing that each left prelinear RL is distributive.

\begin{corollary} \label{cor: prelinear}
  The following are equivalent for each RL $\alg{A}$:
\begin{enumerate}[(i)]
\item $\alg{A}$ is left prelinear, i.e.\ it satisfies $(\one \wedge x \bs y) \vee (\one \wedge y \bs x) \equals \one$.
\item $\alg{A}$ satisfies $\one \wedge (x \vee y) \equals (\one \wedge x) \vee (\one \wedge y)$ and $(x \bs y) \vee (y \bs x) \geq \one$.
\item $\alg{A}$ is distributive and satisfies $(x \bs y) \vee (y \bs x) \geq \one$.
\end{enumerate}
\end{corollary}

\begin{proof}
  (iii) $\Rightarrow$ (ii) $\Rightarrow$ (i): trivial. (i) $\Rightarrow$ (iii): this implication follows immediately from Theorem~\ref{thm: pre-k}, but it is also easy to prove directly. In~\cite[Proposition~6.10]{RLpaper} it was shown that left prelinearity implies the equation
\begin{align*}
  & x \bs (y \vee z) \equals (x \bs y) \vee (x \bs z).
\end{align*}
  (More precisely, the last two steps of the proof of~\cite[Proposition~6.10]{RLpaper} can be contracted to a single step which relies only on left prelinearity.) The equation
\begin{align*}
  & x \bs (y \wedge z) \equals (x \bs y) \wedge (x \bs z),
\end{align*}
  in contrast, holds in all RLs.

  The equation $(\one \wedge (x \bs y)) \vee (\one \wedge (y \bs x)) \geq \one$ clearly implies that $(x \bs y) \vee (y \bs x) \geq \one$. Now let $d \assign a \wedge (b \vee c)$. Then $d \bs a \geq \one$ and by the equations displayed above also
\begin{align*}
  d \bs b = (a \bs b) \vee ((b \bs b) \wedge (c \bs b)) \geq (b \bs b) \wedge (c \bs b) \geq \one \wedge (c \bs b).
\end{align*}
  Similarly, $d \bs c \geq \one \wedge (b \bs c)$. Therefore
\begin{align*}
  d \bs ((a \wedge b) \vee (a \wedge c)) = (d \bs a \wedge d \bs b) \vee (d \bs a \wedge d \bs c) \geq (\one \wedge c \bs b) \vee (\one \wedge b \bs c) = \one,
\end{align*}
  again using the equations displayed above.
\end{proof}

\begin{corollary} \label{cor: preconic}
  The following are equivalent for each RL $\alg{A}$:
\begin{enumerate}[(i)]
\item $\alg{A}$ is left preconic, i.e.\ it satisfies $(\one \wedge x \bs \one) \vee (\one \wedge x) \equals \one$.
\item $\alg{A}$ satisfies $\one \wedge (x \vee y) \equals (\one \wedge x) \vee (\one \wedge y)$ and $(x \bs \one) \vee x \geq \one$.
\item $\alg{A}$ has a semiconic pointed lattice reduct and satisfies $(x \bs \one) \vee x \geq \one$.
\end{enumerate}
\end{corollary}

\begin{proof}
  This again follows immediately from Theorem~\ref{thm: pre-k}.
\end{proof}

\section{Unital lattice subreducts of integral cancellative CRLs}
\label{section: cancellative}

  In this section, we turn to the problem of describing the unital lattice subreducts of integral cancellative CRLs. A residuated lattice is \emph{cancellative} if its monoidal reduct is cancellative, i.e.\ if it satisfies the quasi-equations
\begin{align*}
  & x \cdot y \equals x \cdot z \implies y \equals z, & & y \cdot x \equals z \cdot x \implies y \equals z.
\end{align*}
  These are equivalent to the equations
\begin{align*}
  & x \bs (x \cdot y) \equals y, & & (y \cdot x) / x \equals y.
\end{align*}
  For example, the integers $\Z$ form a cancellative CRL with the standard order and with addition as the monoidal operation, and with the residual $x \rightarrow z \assign -x + z$. The non-positive integers $\Zm$ form an (integral) cancellative CRL with the standard order and addition as the monoidal operation, and with the residual $x \rightarrow z \assign \min(-x + z, \zero)$. In contrast to $\Z$ and $\Zm$, equipping the non-negative integers $\Zp$ with the same order and multiplication yields a commutative cancellative totally ordered monoid which does not form a residuated lattice, since for $x \nleq z$ in $\Zp$ there is no $y$ in $\Zp$ such that $x \cdot y \leq z$.

  The variety of cancellative RLs was first studied in depth by Bahls et al.~\cite{CanRLs}. Among other results, they proved that every lattice is a subreduct of some (simple) integral cancellative RL~\cite[Theorem~4.3]{CanRLs}. The problem of whether each lattice is a subreduct of an (integral) \emph{commutative} cancellative RL was left open~\cite[Problem~8.2]{CanRLs}. We settle this question in the affirmative below (Corollary~\ref{cor: cancellative}). This is equivalent to showing that this variety of (integral) cancellative CRLs does not satisfy any universal sentence, or equivalently any quasi-equation, in the signature of lattices beyond those satisfied by all lattices.

  Recall that a unary operation $\sigma$ on a poset is an \emph{interior operator} if it is order-preserving, idempotent (i.e.\ $\sigma(\sigma(x)) = \sigma(x))$, and decreasing (i.e.\ $\sigma(x) \leq x$). An interior operator on a residuated lattice $\alg{A}$ is a \emph{conucleus} if $\sigma(x) \cdot \sigma(y) \leq \sigma(x \cdot y)$ for all $x, y \in \alg{A}$ and $\sigma(\one) =  \one$. Equivalently, a conucleus on $\alg{A}$ is an interior operator whose image is a submonoid of $\alg{A}$. Crucially, the image $\alg{A}_{\sigma}$ of a conucleus $\sigma$ can be equipped with the structure of a residuated lattice such that joins, products, and the monoidal unit of $\alg{A}_{\sigma}$ are the same as in $\alg{A}$ (see e.g.~\cite{Conuclei}).

  While we shall not need this fact in the following, let us recall here that part of the interest of (integral) cancellative CRLs lies in their close relation with Abelian lattice-ordered groups: they are precisely the conuclear images of (negative cones of) Abelian lattice-ordered groups~\cite[Lemma~4.4]{Conuclei}.

\begin{lemma}
  Each prime-pointed unital lattice is a pointed lattice subreduct of some simple integral cancellative CRL.
\end{lemma}

\begin{proof}
  Consider a non-trivial prime-pointed unital lattice. Recall that such a unital lattice has the form $\alg{L} \oplus \one$ for some lattice $\alg{L}$. If $\alg{L}$ is trivial, then $\alg{L} \oplus \one$ is a subreduct of $\Zm$, we may therefore assume that $\alg{L}$ is non-trivial. Because each lattice embeds into a complete lattice, we may further assume that $\alg{L}$ is complete. Each complete lattice is the image of an interior operator on some powerset: this is an order dual version of the known fact that each complete lattice is the lattice of closed sets of a closure system on some set. In particular, we shall treat $\alg{L}$ as the image of a poset of the form $\{ -2, -1 \}^{X}$ (with the standard componentwise order) for some set $X$ under an interior operator $\sigma$. We may finally assume without loss of generality that the top element of $\alg{L}$ is the constant function with value $-1$.

  Consider the integral cancellative CRL $\Zm^{X}$. We use $\zero$, $\minusone$, $\minustwo$ to denote the appropriate constant functions in $\Zm^{X}$. Let $\alg{A} \leq \Zm^{X}$ be the subalgebra consisting of bounded functions, i.e.\ $f \in \alg{A}$ if and only if there is some $-n \in \Zm$ such that $f(x) \geq -n$ for each $x \in X$. That is, $f \in \alg{A}$ if and only if $f \geq \minusone + \minusone + \dots + \minusone + \minusone$ for some long enough sum. (Clearly $\alg{A}$ is a sublattice and a submonoid. It is closed under residuation because $a \rightarrow b \geq b$ in each integral RL.)

  We take $\alg{M}$ to be the submonoid of $\alg{L}$ with the universe $\down (\minustwo) \cup L \cup \{ \zero \}$. (This is a submonoid because for $a, b \in L$ we always have $a, b \leq \minusone$ and thus $a + b \leq \minustwo$.) The partially ordered submonoid $\alg{M}$ is the reduct of a residuated lattice due to being the image of the interior operator~$\tau$ on $\alg{L}$ such that $\tau(\zero) \assign \zero$, $\tau(a) \assign \sigma(a \wedge \minusone)$ for $a \in \up (\minustwo)$ other than $a = \zero$, and $\tau(a) \assign a \wedge \minustwo$ for $a \notin \up (\minustwo)$. The operator $\tau$ is a conucleus on $\alg{L}$ due to its image being a submonoid of $\alg{L}$. By the definition of $\alg{M}$, the unital lattice $\alg{L} \oplus \one$ is then a unital lattice subreduct of the RL expansion of~$\alg{M}$. Clearly this expansion is integral, commutative, and cancellative, since these properties of RLs depend only on the partially ordered monoid reduct.

  Finally, the residuated lattice expansion of $\alg{M}$ is simple: congruences of integral CRLs are in bijective correspondence with upsets containing $\one$ and closed under multiplication. But each such upset other than $\{ \zero \}$ must contain $\minusone$ (recall the assumption that $\minusone \in L$) and for each $f \in \alg{A}$ we have $f \geq \minusone + \minusone + \dots + \minusone + \minusone$ for some long enough sum.
\end{proof}

\begin{theorem} \label{thm: cancellative}
  The unital lattice subreducts:
\begin{enumerate}[(i)]
\item of f.s.i.\ (simple) integral cancellative CRLs are precisely the prime-pointed unital lattices,
\item of (semisimple) integral cancellative CRLs are precisely the semi-prime-pointed unital lattices.
\end{enumerate}
\end{theorem}

\begin{proof}
  This follows from the above lemma and the fact that each f.s.i.\ CRL is prime-pointed (Theorem~\ref{thm: main}).
\end{proof}

\begin{corollary} \label{cor: cancellative}
  Every lattice is a subreduct of an integral cancellative CRL.
\end{corollary}

  The above corollary can also be derived from the unpublished result of Young~\cite{Young} stating that intervals of the form $[u, \one]$ in integral cancellative CRLs (equipped with a suitable CRL structure) are precisely the conuclear images of MV-algebras. Indeed, the reader familiar with MV-algebras will easily observe that our construction relies precisely on showing that every lattice is a conuclear image of an MV-algebra.

\section*{Declarations}

\subsection*{Conflict of interest} The author has no conflict of interest to declare.

\subsection*{Funding} This work was supported by the Beatriu de Pin\'{os} grant 2021 BP 00212 of the grant agency AGAUR of the Generalitat de Catalunya.

\subsection*{Ethical approval} Not applicable.

\subsection*{Availability of data and materials} Not applicable.


\begin{thebibliography}{10}

\bibitem{JKP}
Kira Adaricheva, Mikl\'{o}s Mar\'{o}ti, Ralph McKenzie, J.B. Nation, and
  Eric~R. Zenk.
\newblock The {J\'{o}nsson}-{Kiefer} property.
\newblock {\em Studia Logica}, 83(1):111--131, 2006.

\bibitem{CanRLs}
Patrick Bahls, Jac Cole, Nikolaos Galatos, Peter Jipsen, and Constantine
  Tsinakis.
\newblock Cancellative residuated lattices.
\newblock {\em Algebra Universalis}, 50(1):83--106, 2003.

\bibitem{Bergman}
Clifford Bergman.
\newblock {\em Universal algebra: fundamentals and selected topics}, volume 151
  of {\em Studies in Logic and the Foundations of Mathematics}.
\newblock Elsevier, 2007.

\bibitem{RLpaper}
Kevin Blount and Constantine Tsinakis.
\newblock The structure of residuated lattices.
\newblock {\em International journal of Algebra and Computation},
  13(4):437--461, 2003.

\bibitem{CrawleyDilworth}
Peter Crawley and Robert~P. Dilworth.
\newblock {\em Algebraic theory of lattices}.
\newblock Prentice-Hall, 1973.

\bibitem{CD}
Janusz Czelakowski and Wies{\l}aw Dziobiak.
\newblock Congruence distributive quasivarieties whose finitely subdirectly
  irreducible members form a universal class.
\newblock {\em Algebra universalis}, 27:128--149, 1990.

\bibitem{ConicFG}
Wesley Fussner and Nick Galatos.
\newblock Semiconic idempotent logic {I}: structure and local deduction
  theorems.
\newblock Preprint available at \url{https://arxiv.org/abs/2208.09724v2}.

\bibitem{GalatosHorcik}
Nikolaos Galatos and Rostislav Hor\v{c}\'{\i}k.
\newblock Cayley's and {Holland's} theorems for idempotent semirings and their
  applications to residuated lattices.
\newblock {\em Semigroup Forum}, 87(3):569--589, 2013.

\bibitem{Glimpse}
Nikolaos Galatos, Peter Jipsen, Tomasz Kowalski, and Hiroakira Ono.
\newblock {\em Residuated lattices: an algebraic glimpse at substructural
  logics}, volume 151 of {\em Studies in Logic and the Foundations of
  Mathematics}.
\newblock Elsevier, 2007.

\bibitem{HorcikFEP}
Rostislav Hor\v{c}\'{\i}k.
\newblock Finite embeddability property for residuated lattices via regular
  languages.
\newblock In Nikolaos Galatos and Kazushige Terui, editors, {\em Hiroakira Ono
  on Substructural Logics}, volume~23 of {\em Outstanding Contributions to
  Logic}, pages 273--298. Springer, 2022.

\bibitem{ConicHR}
Ai-ni Hsieh and James~G. Raftery.
\newblock Semiconic idempotent residuated structures.
\newblock {\em Algebra universalis}, 61:413--430, 2009.

\bibitem{Jonsson}
Bjarni J\'{o}nsson.
\newblock Congruence distributive varieties.
\newblock {\em Mathematica japonica}, 42(2):353--401, 95.

\bibitem{JonssonKiefer}
Bjarni J\'{o}nsson and James~E. Kiefer.
\newblock Finite sublattices of a free lattice.
\newblock {\em Canadian Journal of Mathematics}, 14:487--497, 1962.

\bibitem{RLBook}
George Metcalfe, Francesco Paoli, and Constantine Tsinakis.
\newblock {\em Residuated structures in algebra and logic}, volume 277 of {\em
  Mathematical Surveys and Monographs}.
\newblock American Mathematical Society, 2023.

\bibitem{Conuclei}
Franco Montagna and Constantine Tsinakis.
\newblock Ordered groups with a conucleus.
\newblock {\em Journal of Pure and Applied Algebra}, 214(1):71--88, 2010.

\bibitem{Stern}
Manfred Stern.
\newblock {\em Semimodular lattices: Theory and applications}.
\newblock Number~73 in Encyclopedia of Mathematics and its Applications.
  Cambridge University Press, 1999.

\bibitem{Young}
William Young.
\newblock Heyting algebras as intervals of commutative, cancellative residuated
  lattices.
\newblock Unpublished manuscript, 2014.

\end{thebibliography}
\end{document}